\documentclass[11pt]{article}
\usepackage{graphicx} 
\usepackage{amsthm, amsmath, amssymb}
\usepackage{xcolor}
\usepackage[textsize=footnotesize]{todonotes}
\usepackage{multirow}
\usepackage{comment}
\theoremstyle{plain}
\newtheorem{theorem}{Theorem}[section]
\newtheorem{lemma}[theorem]{Lemma}
\newtheorem{proposition}[theorem]{Proposition}
\newtheorem{corollary}[theorem]{Corollary}

\newcommand\ord{{\rm ord}}
\newcommand{\Stab}{{\rm Stab}}

\newtheorem{definition}[theorem]{Definition}

\newtheorem{example}[theorem]{Example}
\newtheorem{remark}[theorem]{Remark}

\newcommand{\Znk}{\mathbb{Z}_{2^k}}
\newcommand{\Znl}{\mathbb{Z}_{2^\ell}}
\newcommand{\V}{\mathbb{V}}  
\newcommand{\F}{\mathbb{F}_{2}}
\newcommand{\Q}{\mathbb{Q}}
\newcommand{\Z}{\mathbb{Z}}
\newcommand{\N}{\mathbb{N}}
\newcommand{\C}{\mathbb{C}}
\newcommand{\R}{\mathbb{R}}
\newcommand{\WHT}{\mathcal{H}} 
\newcommand{\zetaa}[1]{\zeta_{2^{#1}}} 

  \newcommand\supp{{\rm supp}}
 
\renewcommand{\to}{\rightarrow}

\usepackage[margin=1.25in]{geometry} 

\title{\bf Overconstrained character sums over finite abelian groups and decompositions of generalized bent, plateaued and landscape functions\footnote{This research was initiated during the workshop
Bent Camp ``Bent Functions: Recent Developments" 26.01.2025 - 01.02.2025 at Sabanci University Istanbul, Turkey.}}
\author{\Large Ay\c{c}a \c{C}e\c{s}melio\u{g}lu\thanks{Faculty of Engineering, \"{O}zy\u{e}\u{g}in University, 34794, \c{C}ekmek\"{o}y-\.{I}stanbul, Turkey; Email:  \tt{ayca.cesmelioglu@ozyegin.edu.tr}},
Constanza Riera\thanks{Department of Computer Science, Electrical Engineering and Mathematical Sciences, Western Norway University of Applied Sciences, 5020 Bergen, Norway; Email:  \tt{csr@hvl.no}}, Pantelimon St\u{a}nic\u{a}\thanks{Applied Mathematics Department, Naval Postgraduate School, Monterey,  CA 93943, USA; Email:  \tt{pstanica@nps.edu}}} 
\date{\today}

\begin{document}

\maketitle

\begin{abstract}
Generalized bent (gbent) functions from an $n$-variable Boolean space to the ring $\mathbb{Z}_{2^k}$ are fundamental objects in cryptography and sequence design. While typically studied through their binary components, we introduce a more general $2^\ell$-adic representation for integers $\ell, r$ with $k = \ell r$, viewing such functions as linear combinations of $r$ component functions valued in $\mathbb{Z}_{2^\ell}$.
We establish a foundational result on overconstrained character sums over finite
abelian groups, proving unconditionally that sequences whose Fourier transforms have
two-level magnitude spectra must be extremely sparse under a common-argument
hypothesis, with a conditional extension to multi-level spectra.  We apply this result to generalized plateaued functions under some conditions.
We also prove that, under 
the $2^\ell$-adic decomposition, if $f: \mathbb{F}_2^n \to \mathbb{Z}_{2^k}$
is landscape (which includes generalized plateaued and  generalized bent functions), then every function in a specific affine space over $\mathbb{Z}_{2^\ell}$
is landscape with identical Walsh magnitudes, an unconditional necessity result that
requires no structural assumptions on $f$. A complete characterization via a small subset of these maps
is also established. Sufficiency holds also for  generalized bent functions and generalized plateaued via  linear combinations of the lower components under natural structural assumptions; a counterexample demonstrates that assumptions such as these are genuinely necessary. 
For landscape functions in general, this method reduces verification from $2^{2^{k-1}}$ checks to less than $2^{k-\ell+1}+1$ conditions;  in the case of generalized bents this can be further reduced to just  $1$ basis function under a common-argument hypothesis, and in the case of generalized plateaued functions, with some additional assumptions,  the verification can be reduced to $2^{k-\ell}$. 
The $2^\ell$-adic characterization extends to landscape functions and preserves key cryptographic properties such as 
duality, and differential uniformity.  Our unified approach combines Galois theory of cyclotomic extensions with additive combinatorics to reveal hierarchical structure in generalized Boolean functions, offering both theoretical insights and practical verification algorithms for cryptographic design.
\end{abstract}



\section{Introduction and definitions}
\label{sec:intro}

Throughout this paper, $\F$ denotes the binary field, $\V = \F^n$ the $n$-dimensional
vector space over $\F$, and $\Znk$ the ring of integers modulo $2^k$. We write
$\zeta_{m} = e^{2\pi i/m}$ for a primitive $m$-th root of unity. The inner product
on $\V$ is denoted $\langle u, x \rangle = \sum_{i=1}^n u_i x_i$, and we use
$\WHT_f$ for the Walsh-Hadamard transform of $f$ (see Definition~\ref{defn:WHT}).

Functions with low and uniform Walsh-Hadamard spectra play a crucial role in the
design of cryptographic primitives like S-boxes and in the construction of sequences
with ideal correlation properties. Bent functions, introduced by Rothaus~\cite{rothaus},
are Boolean functions that achieve the theoretical minimum for the magnitude of their
Walsh-Hadamard transform, making them maximally nonlinear. For a Boolean function
$f: \F^n \to \F$ on $n$ variables (where $n$ is even), bent functions satisfy
$|\WHT_f(u)| = 2^{n/2}$ for all $u \in \F^n$, achieving the theoretical lower bound
on the maximum Walsh spectrum magnitude. The concept has been generalized in several
important directions. Kumar, Scholtz, and Welch~\cite{ksw} introduced generalized bent
(gbent) functions, which map from an $n$-variable Boolean space $\V = \F^n$ to the
ring of integers modulo $2^k$, denoted $\Znk$. These functions retain the uniform
Walsh-Hadamard spectrum property, $|\WHT_f(u)| = 2^{n/2}$ for all $u$. Bent functions
have found applications in coding theory, cryptography, and the construction of
sequences with ideal autocorrelation properties~\cite{Budaghyan2014,carlet_survey,CS17,Mesnager2016,mms}.

A function $f: \V \to \Znk$ is typically analyzed through its \emph{binary
decomposition}
\[
f(x) = a_0(x) + 2a_1(x) + \cdots + 2^{k-1} a_{k-1}(x),
\]
where each $a_i: \V \to \F$ is a Boolean function. A fundamental result by Mesnager
et al.~\cite{mtqwwf} provides a complete characterization: $f$ is gbent if and only
if every Boolean function in the affine space
\[
\{a_{k-1} + F(a_0, \ldots, a_{k-2}) \mid F : \mathbb{F}_2^{k-1} \to \mathbb{F}_2\}
\]
is bent. Here, $F$ ranges over all Boolean functions from $\mathbb{F}_2^{k-1}$ to
$\mathbb{F}_2$, and $F(a_0, \ldots, a_{k-2})$ denotes the composition
$x \mapsto F(a_0(x), \ldots, a_{k-2}(x))$. This characterization has $2^{2^{k-1}}$
conditions (one for each $F$).

While the binary decomposition provides a complete characterization, it treats $f$
as a tower of binary extensions. In this work, we investigate a more general
structural decomposition that reveals additional structure. For integers $\ell, r$
such that $k = \ell r$, we consider the \emph{$2^\ell$-adic decomposition}
\[
f(x) = c_0(x) + 2^{\ell}c_1(x) + 2^{2\ell}c_2(x) + \cdots + 2^{(r-1)\ell}c_{r-1}(x),
\]
where each component function $c_j: \V \to \Znl$ is a generalized function to a
smaller ring. This representation is natural for several reasons. When $k=\ell r$ is
composite, the $2^\ell$-adic decomposition respects the tower structure
$\Znk \supset 2^\ell \Znk \supset \cdots \supset 2^{(r-1)\ell}\Znk$, viewing $f$
as built from $r$ layers each operating over $\Znl$. In applications where arithmetic
modulo $2^\ell$ is more natural than bit operations (e.g., working with 4-bit or
8-bit words), the $2^\ell$-adic view aligns with implementation constraints.
Different decompositions may also reveal different structural properties, providing
an alternative lens for analyzing security. The $2^\ell$-adic decomposition naturally
extends to $p$-adic representations, potentially connecting to deeper
number-theoretic structures.

Our main structural result on character sums (Theorem~\ref{thm:overconstrained})
proves that sequences whose Fourier transforms have two-level magnitude spectra must
be extremely sparse when a common-argument hypothesis holds, with a conditional
extension to multi-level spectra (Theorem~\ref{thm:overconstrained_multilevel}).
The key technical ingredient is Theorem~\ref{thm:two_point_analysis}, which
establishes that two-point character sums can have two-level spectra only when the order of the elements $r=3$
under an exceptional algebraic condition, or are trivially one-dimensional.
Furthermore, we show 
unconditional and conditional characterizations of a function in terms of its $2^\ell$-adic decomposition: our main results here are unconditional necessity and sufficiency results:
An unconditional necessity result, stating that if $f: \V \to \Znk$ is a
landscape function (which includes generalized plateaued and  generalized bent functions), then every function in the affine space generated by its
$2^\ell$-adic components is a landscape function from $\V$ to $\Znl$ with identical
Walsh magnitudes (Theorem~\ref{thm:landscape_necessity}), and an unconditional sufficiency condition for landscape functions 
via a small subset of maps $F: (\Znl)^{r-1} \to \Z_{2^m}$, for any $1\leq m\leq \ell$, is established (Theorem~\ref{thm:landscape_iff}). The latter provides with a verification method  that is exponentially more
efficient than the binary approach (less than $2^{k-\ell+1}+1$ checks over $\Znl$ instead of
$2^{2^{k-1}}$ over $\Znk$).  
A fundamentally different sufficiency condition when the nonzero partition coefficients
satisfy the common-argument hypothesis and an additional condition (Theorem~\ref{thm:plateaued_common_arg} for
gplateaued functions) is also given. Example~\ref{ex:counterexample_3to1} demonstrates that
these assumptions are genuinely necessary. For gbent functions specifically we state and prove Proposition ~\ref{thm:iff_basis} and 
Corollary~\ref{cor:main},  providing both a systematic method for constructing
families of gbent functions and a verification method (under the assumption of the common-argument hypothesis) that is exponentially even more efficient than the binary approach: $r-1<k-1$ checks over $\Znl$ instead of
$2^{2^{k-1}}$, reducible to just $1$ basis function. 

We also demonstrate that the $2^\ell$-adic decomposition preserves or has predictable
interactions with critical cryptographic properties: the Maiorana--McFarland class
structure (Theorem~\ref{thm:MM}), duality (Proposition~\ref{prop:dual}),
quadraticity (Corollary~\ref{cor:second_order}), and differential uniformity
(Theorem~\ref{thm:differential_affine}). Finally, we perform a computational analysis of a few classical cipher
S-boxes  via our $2^\ell$-adic perspective, and  propose some open questions.

The paper is organized as follows. Section~\ref{sec:prelim} provides necessary
background on generalized bent functions, Walsh-Hadamard transforms, and related
definitions. Section~\ref{sparsity} presents and proves the main results on overconstrained character sums: the unconditional two-level theorem (Theorem~\ref{thm:overconstrained}) and its conditional multi-level extension
(Theorem~\ref{thm:overconstrained_multilevel}). Section~\ref{sec:main} presents
the $2^\ell$-adic decomposition theory for landscape, gbent, and gplateaued functions, including the unconditional necessity theorem, the complete
characterization, sufficiency results under structural assumptions, a counterexample demonstrating their necessity, and verification complexity analysis.
Section~\ref{sec:properties} explores whether some structural properties are preserved by the decomposition: duality, the Maiorana--McFarland class, derivatives, and differential
spectra. Section~\ref{sec:classical-sboxes} presents the analysis of classical cipher S-boxes. Finally, Section~\ref{sec:conclusion} concludes with a discussion
of open problems and future directions.

\section{Preliminaries}
\label{sec:prelim}

Beyond the foundational work of Rothaus~\cite{rothaus} on bent functions and Kumar
et al.~\cite{ksw} on gbent functions, several lines of research inform this work.
The complete characterization of gbent functions in terms of their binary components
was established by Mesnager et al.~\cite{mtqwwf} (see also~\cite{mmms,mms}). This seminal result showed that the gbent property can be verified by checking that all functions in a certain affine space are bent. Our work provides an analogous result for more general decompositions. Mesnager, Tang, and Qi~\cite{MTQ17} extended the study to generalized plateaued functions, which have Walsh-Hadamard spectra taking on at most two values (one of which is zero), and provided a characterization analogous to~\cite{mtqwwf} for this broader class. Our work extends and generalizes
both of these characterizations to the $2^\ell$-adic setting, separating the unconditional necessity direction from the sufficiency direction which requires
additional structural assumptions.
For more on bent and other properties of cryptographic Boolean functions, the reader can consult standard references like~\cite{Budaghyan2014,carlet_survey, CS17, Mesnager2016}

The Maiorana-McFarland (M-M) class is one of the most important classes of bent and
gbent functions due to its explicit construction and rich algebraic structure.
Functions in this class have the form $f(x,y) = \langle x, \pi(y) \rangle + g(y)$
where $\pi$ is a permutation. We show that our decomposition has predictable behavior under this structure
(Theorem~\ref{thm:MM}), enabling the systematic study of M-M gbent functions through
their $2^\ell$-adic components. Martinsen et al.~\cite{mms} studied Gray images of
generalized bent functions, providing another perspective on the relationship between
Boolean and generalized bent functions. Their approach uses the Gray map to connect
$\Z_4$ functions to pairs of Boolean functions, which can be seen as a special case
of our analysis when $\ell = 2, r = 2$. Riera and St\u{a}nic\u{a}~\cite{RS19}
introduced landscape Boolean functions, which include generalized bent and plateaued functions and can have more than two distinct Walsh
spectrum magnitudes (but still exhibit regular structure). We extend this concept to
the generalized setting and investigate their $2^\ell$-adic structure. The
differential uniformity of functions is a crucial cryptographic measure. While our
focus is on spectral properties, we investigate how the $2^\ell$-adic decomposition
interacts with differential behavior (Theorem~\ref{thm:differential_affine}),
complementing spectral analysis.

\begin{definition}
\label{defn:WHT}
The Walsh-Hadamard transform (WHT) of a function $f: \V \to \Znk$ at $u \in \V$ is
\[
\WHT_f(u) = \sum_{x \in \V} \zetaa{k}^{f(x)} (-1)^{\langle u, x \rangle}.
\]

If $k=1$, this reduces to the classical Walsh-Hadamard transform.
\end{definition}

With this (unnormalized) Walsh--Hadamard transform, the inversion and Parseval
identities read
\[
f(x)=2^{-n}\sum_{\omega\in\mathbb{F}_2^n}\WHT_f(\omega)\,(-1)^{\langle \omega,x\rangle},
\text{ and }
\sum_{x\in\mathbb{F}_2^n}|f(x)|^2
=
2^{-n}\sum_{\omega\in\mathbb{F}_2^n}|\WHT_f(\omega)|^2=2^{n}
\]

A function $f: \V \to \Znk$ is \emph{generalized bent} (gbent) if
$|\WHT_f(u)| = 2^{n/2}$ for all $u \in \V$. All gbent functions are
\emph{regular}, meaning $\WHT_f(u) = 2^{n/2}\zetaa{k}^{\widetilde{f}(u)}$
for a unique function $\widetilde{f}: \V \to \Znk$, called the \emph{dual} of $f$ (also gbent~\cite{ksw}), except when $n$ is odd and $k=2$, in which case $\WHT_f(u) = 2^{(n-1)/2}(\pm 1 \pm i)$ and no dual is defined.

We extend this notion to possibly two or three Walsh coefficients. Let $n, s$ be
integers such that $0 \le s \le n$ and $n \equiv s \pmod{2}$. A function
$f: \V \to \Znk$ is \emph{$s$-generalized plateaued} ($s$-gplateaued)~\cite{MTQ17}
if for all $u \in \V$, $|\WHT_f(u)| \in \{0, 2^{(n+s)/2}\}$. Under this definition,
the parameter $s$ relates to the dimension of the subspace on which the function's
derivative is zero. A gbent function is the special case where $s=0$, since its
Walsh spectrum magnitude is uniformly $2^{n/2}$.

Let $n=2m$. A function $f: \F^m \times \F^m \to \Znk$ is of the
\emph{Maiorana-McFarland (M-M) class} if it has the form
$f(x,y) = \langle x, \pi(y) \rangle + g(y)$, where $\pi: \F^m \to \F^m$ is a
permutation and $g: \F^m \to \Znk$ is any function. By~\cite{ksw}, all functions
in the Maiorana-McFarland class are gbent.

Every function $f: \V \to \Znk$ has a unique \emph{binary decomposition}
$f(x) = \sum_{i=0}^{k-1} a_i(x) 2^i$, where each $a_i: \V \to \F$ is a Boolean
function, called a \emph{binary coordinate} of $f$.

Recall that the \emph{$t$-order derivative} of a function $f: \V \to \Znk$ in the
directions $a_1,\ldots,a_t \in \V$ is the function
\[
D_{a_1,\ldots,a_t}f(x) = \sum_{S \subseteq [t]} (-1)^{|S|} f\left(x + \sum_{i \in S}
a_i\right) \pmod{2^k}.
\]
For $t=2$,
$D_{a,b}f(x) = f(x) - f(x+a) - f(x+b) + f(x+a+b) \pmod{2^k}$.
Recall that if a function is \emph{quadratic} then $D_{a,b}f$ is constant for all
$a,b \in \V$. Over the integers, a function is quadratic if its algebraic normal
form contains only monomials of degree at most $2$ (with that degree attained).
The second-order derivative characterization extends this to the modular setting.

The \emph{differential} of a function $f: \V \to \Znk$ is
$\Delta_f(a,b) = |\{x \in \V \mid f(x+a) - f(x) = b \pmod{2^k}\}|$.
The multiset $\{\Delta_f(a,b) \mid a \in \V \setminus \{0\}, b \in \Znk\}$ is called
the \emph{differential spectrum} of $f$. The function $f$ is called
\emph{differentially $\delta$-uniform} if $\max_{a \neq 0, b} \Delta_f(a,b) = \delta$.

Mesnager et al.~\cite{mtqwwf} established the following binary component
characterization, which motivated our work and is the starting point for the
$2^\ell$-adic generalization developed in Section~\ref{sec:main}.

\begin{theorem}
\label{thm:binary_char}
Let $f: \V \to \Znk$ have binary components $a_0, \ldots, a_{k-1}$. Then $f$ is
gbent if and only if for every Boolean function $F: \F^{k-1} \to \F$, the function
$a_{k-1} + F(a_0, \ldots, a_{k-2})$ is a bent function from $\V$ to $\F$.
\end{theorem}

 We call a function $f: \V \to \Znk$ a \emph{landscape} function (notion introduced by Riera and St\u{a}nic\u{a}~\cite{RS19}) if there exist
$t\geq 1$, $m_i\in\N_0=\{0,1,\ldots\}$ (natural numbers), $v_i\in2\N_0+1$, $1\leq i\leq t$, such that
\[
\{|\WHT_f(u)|\}_{u \in\supp(\WHT)} = \{2^{m_1/2}v_1,\ldots,2^{m_t/2}v_t\}.
\]
We call the set of pairs $\{(m_1,v_1),(m_2,v_2),\ldots\}$ the \emph{levels} of $f$, and $t+1$ (if $0$ belongs to the Walsh-Hadamard spectrum), or $t$ (if $0$ is not in the spectrum) the \emph{length} of $f$. Certainly, every classical Boolean function is a landscape function. That is not true for generalized Boolean ($k>1$), as the Walsh-Hadamard values are $\pm$ sums of powers of the primitive root, so the moduli of the spectra values
may contain elements outside $\Z\cup\sqrt{2}\,\Z$; however, gplateaued (which includes generalized bent/semibent) functions are all examples of landscape functions, which are all  somewhat regular.

\begin{theorem}[\cite{RS19}]
\label{reg-landscape}
Let $f: \V \to \Znk$, $k\geq 1$, be a landscape function, and
$\zeta=e^{2\pi i/2^k}$ be a $2^k$-primitive root of unity. Let
$u \in\supp(\WHT_f)$, with $|\WHT_f(u)|=2^{m/2}v$, $m \in\N_0$,
$v\in2\N_0+1$. Then, if $m$ is even, or $m$ is odd and $k>2$, we have $\WHT_f(u)=2^{m/2}v\zetaa{k}^{f^*(u)}$, for some value $f^*(u)\in\Znk$. If $m$ is odd and $k=2$, and $a_0\neq0$,
$\WHT_f(u) = 2^{\lfloor{m/2}\rfloor}v (\epsilon_1+\epsilon_2 i)$, with $\epsilon_1,\epsilon_2\in\{\pm 1\}$, with the additional possibility, if $v$ is the
largest component of a Pythagorean triple $v_1^2+v_2^2=v^2$, of
$\WHT_f(u) = 2^{\lfloor{m/2}\rfloor}\left(v_1\epsilon_1+v_2\epsilon_2\pm i
(v_1\epsilon_2-v_2\epsilon_1)\right)$.
If $m$ is odd and $k=2$, there is no function $f$ with $a_0=0$ such that $|\WHT_f(u)|=2^{m/2}v$. If $m$ is odd and $k=1$, there is no function $f$ such
that $|\WHT_f(u)|=2^{m/2}v$, $m \in\N_0$, $v\in2\N_0+1$.
\end{theorem}

 The binary component characterization of $s$-gplateaued functions established in~\cite{MTQ17} --- that $f$ is $s$-gplateaued if and only if $a_{k-1} + F(a_0, \ldots, a_{k-2})$ is $s$-plateaued for every Boolean function
$F: \F^{k-1} \to \F$ --- is the direct precursor to our $2^\ell$-adic generalization in Section~\ref{sec:main}, where we separate the necessity and sufficiency directions and identify the precise structural assumptions needed for each.

\section{Structure and sparsity of constrained character sums}
\label{sparsity}

This section establishes foundational results on character sums whose magnitudes are constrained to lie in a finite set. We prove unconditionally that two-level magnitude spectra force extreme sparsity under the common-argument hypothesis (Theorem~\ref{thm:overconstrained}), with a conditional extension to multi-level spectra (Theorem~\ref{thm:overconstrained_multilevel}). These results are essential for the characterization of generalized bent and plateaued functions via $2^\ell$-adic decomposition.

We work throughout with a finite abelian group $G$, its character group $\widehat{G}$, and complex-valued functions on $G$. The key insight is that if the Fourier transform of a sequence takes values in a small set of magnitudes, then the sequence itself must be highly structured -- often supported on a single point.

\subsection{Preliminaries: additive combinatorics and Fourier analysis}

We begin with two fundamental tools. 

\begin{theorem}[Kneser {\cite[Theorem 5.5]{TaoVuAddComb}}]
\label{thm:kneser}
Let $G$ be a finite abelian (additive) group and $A, B \subseteq G$ nonempty subsets. Then
$|A + B| \geq |A| + |B| - |\Stab(A+B)|$,
where $\Stab(S) = \{g \in G : S + g = S\}$ is the stabilizer of $S$ and $A+B=\{a+b: a\in A,b\in B\}$.
\end{theorem}

From Kneser’s theorem, we derive growth estimates for iterated sumsets (see~\cite{TaoVuAddComb} for related sumset growth results).

\begin{lemma}
\label{lem:sumset_growth}
Let $G$ be a finite abelian group, $0\in D \subseteq G$, and let $H = \Stab(D + D)$. Let $\pi: G \to G/H$ be the canonical projection, 
and set $\overline{D} = \pi(D) \subseteq G/H$.
Then:
\begin{enumerate}
\item[\textup{(i)}] $\Stab(\overline{D} + \overline{D}) = \{0\}$ in $G/H$.
\item[\textup{(ii)}] For all integers $k \geq 1$,
$|k\overline{D}| \geq \min\left\{|G/H|,\, k(|\overline{D}|-1)+1\right\}$.
\item[\textup{(iii)}] If $|\overline{D}| \geq 2$ and $|k\overline{D}| < |G/H|$, then $|(k+1)\overline{D}| > |k\overline{D}|$.
\end{enumerate}
\end{lemma}

\begin{proof} 
Let $G$ be a finite abelian group, $D \subseteq G$ with $0 \in D$, and let $H = \Stab(D + D)$. Let $\pi: G \to G/H$ be the canonical projection, and set $\overline{D} = \pi(D) \subseteq G/H$.

(i) We need to show $\Stab(\overline{D} + \overline{D}) = \{0\}$ in $G/H$. 
Let $\bar{g} \in \Stab(\overline{D} + \overline{D})$, so that $\overline{D} + \overline{D} + \bar{g} = \overline{D} + \overline{D}$ in $G/H$. 
Choose a representative $g \in G$ with $\pi(g) = \bar{g}$. Then 
$\pi(D + D + g) = \pi(D + D)$.
This means that for every $d_1, d_2 \in D$, there exist $d_1', d_2' \in D$ and $h \in H$ such that 
$d_1 + d_2 + g = d_1' + d_2' + h$, which gives $g \in (D+D) - (D+D) + H$. 
Since this holds for all such pairs, and $H = \Stab(D+D)$ is the maximal subgroup with $D+D+H = D+D$, we conclude that $D+D+g \subseteq D+D+H = D+D$. 
Therefore $g \in \Stab(D+D) = H$, which means $\bar{g} = \pi(g) = 0$ in $G/H$.
Thus $\Stab(\overline{D} + \overline{D}) = \{0\}$, as claimed.

(ii) Since $\Stab(\overline{D} + \overline{D}) = \{0\}$, Kneser's Theorem (Theorem~\ref{thm:kneser}) applied to $A = B = \overline{D}$ gives
$|\overline{D} + \overline{D}| \geq 2|\overline{D}| - 1$.
Now we proceed by induction on $k \geq 1$. The base case $k=1$ is trivial: $|1\cdot\overline{D}| = |\overline{D}| = \min\{|G/H|, 1\cdot(|\overline{D}|-1)+1\}$.
Assume the claim holds for some $k \geq 1$, i.e.,
\[
|k\overline{D}| \geq \min\{|G/H|,\, k(|\overline{D}|-1)+1\}.
\]
If $|k\overline{D}| = |G/H|$, then $(k+1)\overline{D} = G/H$, and the inequality holds trivially. Otherwise, $|k\overline{D}| < |G/H|$, and since $\Stab((k+1)\overline{D}) \subseteq \Stab(\overline{D} + \overline{D}) = \{0\}$, we may apply Kneser's theorem to $A = k\overline{D}$ and $B = \overline{D}$, and get
\[
|(k+1)\overline{D}| = |k\overline{D} + \overline{D}| \geq |k\overline{D}| + |\overline{D}| - 1.
\]
By the induction hypothesis,
\[
|(k+1)\overline{D}| \geq \big(k(|\overline{D}|-1)+1\big) + (|\overline{D}|-1) = (k+1)(|\overline{D}|-1) + 1.
\]
Since $|G/H|\geq|(k+1)\overline{D}|$, then  $|G/H|\geq(k+1)(|\overline{D}|-1) + 1$, and thus
\[
|(k+1)\overline{D}| \geq \min\{|G/H|,\, (k+1)(|\overline{D}|-1) + 1\},
\]
completing the induction.

(iii) Suppose $|\overline{D}| \geq 2$ and $|k\overline{D}| < |G/H|$. Then from the inductive step above,
\[
|(k+1)\overline{D}| \geq |k\overline{D}| + |\overline{D}| - 1 \geq |k\overline{D}| + 1,
\]
since $|\overline{D}| \geq 2$. Hence $|(k+1)\overline{D}| > |k\overline{D}|$, as required.
\end{proof}

We also recall basic facts from Fourier analysis on finite abelian groups~\cite{nathanson,rudin,TaoVuAddComb}. For $f: G \to \mathbb{C}$, the Fourier transform is
\[
\widehat{f}(\chi) = \sum_{x \in G} f(x) \overline{\chi(x)}, \quad \chi \in \widehat{G},
\]
where $\widehat{{G}}$ denote the character group of~$G$.
Key properties~\cite{nathanson,TaoVuAddComb,terras} include
inversion, $f(x) = \frac{1}{|G|} \sum_{\chi \in \widehat{G}} \widehat{f}(\chi) \chi(x)$,
Parseval identity, $\sum_{x \in G} |f(x)|^2 = \frac{1}{|G|} \sum_{\chi \in \widehat{G}} |\widehat{f}(\chi)|^2$, 
convolution, $\widehat{f * g} = \widehat{f} \cdot \widehat{g}$,
and uncertainty principle for finite abelian groups proved by Meshulam~\cite{Meshulam}, inspired by earlier
 work of Donoho and Stark~\cite{DonohoStark} on discrete uncertainty principles.

\begin{theorem}[\cite{DonohoStark,Meshulam}]
\label{thm:uncertainty}
Let $G$ be a finite abelian group and let $0\neq f\in\C[G]$.
Then
\begin{equation}
\label{eq:uncertainty_princ}
|\supp(f)|\cdot|\supp(\widehat f)|\ \ge\ |G|.
\end{equation}
Moreover, let $k:=|\supp(f)|$, and let $d_1\le k\le d_2$ be the consecutive divisors
of $|G|$ such that $d_1$ is maximal with $d_1\le k$ and $d_2$ is minimal with $d_2\ge k$.
Then
\[
|\supp(\widehat f)|\ \ge\ \frac{|G|}{d_1d_2}\,(d_1+d_2-k).
\]
Equality in \eqref{eq:uncertainty_princ} holds if and only if
$f$ is, up to translation, modulation by a character, and multiplication by
a nonzero scalar, the indicator function of a subgroup of $G$.
\end{theorem}

The following definition and three lemmas are rather standard in harmonic analysis on finite abelian groups (see e.g.~\cite[Chapter 1]{terras}), but we provide a few arguments for clarity.
\begin{definition}
Let $H\le G$ be a subgroup and let $\pi:G\to G/H$ be the quotient map.
For $f\in\C[G]$, define the {\em pushforward} $\bar f\in\C[G/H]$ by
\[
\bar f(\bar x)\;:=\;\sum_{x\in \pi^{-1}(\bar x)} f(x),\quad \bar x\in G/H.
\]    
\end{definition}
\begin{lemma} 
\label{lem:pushforward_fourier}
Let $H\le G$ be a subgroup and let $\pi:G\to G/H$ be the quotient map.
For $f\in\C[G]$, 
for every character $\psi\in\widehat{G/H}$ we have
$\widehat{\bar f}(\psi)=\widehat f(\psi\circ\pi)$.
In particular, if $\psi$ is identified with a character of $G$ that is trivial on $H$, then
$\widehat{\bar f}(\psi)=\widehat f(\psi)$.
\end{lemma}

\begin{proof}
We compute
\[
\widehat{\bar f}(\psi)
=\sum_{\bar x\in G/H}\bar f(\bar x)\,\overline{\psi(\bar x)}
=\sum_{\bar x\in G/H}\ \sum_{x\in\pi^{-1}(\bar x)} f(x)\,\overline{\psi(\pi(x))}
=\sum_{x\in G} f(x)\,\overline{(\psi\circ\pi)(x)}
=\widehat f(\psi\circ\pi).
\]
The final claim follows since $\psi\circ\pi$ is exactly the corresponding character of $G$ trivial on~$H$.
\end{proof}

\begin{lemma}
\label{lem:pushforward_convolution}
Let $H\le G$ and let $\pi:G\to G/H$ be the quotient map. For $f,g\in \C[G]$, 
$\overline{f*g}=\bar f * \bar g$ in  $\C[G/H]$,
where $f*g$ indicates the convolution of $f$ and $g$ (below). 
Consequently, for every polynomial $P\in\C[X]$ and $\nu\in\C[G]$ one has $\overline{P(\nu)}=P(\bar\nu)$.
\end{lemma}

\begin{proof}
For $\bar x\in G/H$,
\[
\overline{f*g}(\bar x)
=\sum_{x\in\pi^{-1}(\bar x)} (f*g)(x)
=\sum_{x\in\pi^{-1}(\bar x)}\ \sum_{y\in G} f(y)\,g(x-y)
=\sum_{y\in G} f(y)\sum_{x\in\pi^{-1}(\bar x)} g(x-y).
\]
Since $\pi(x-y)=\bar x-\pi(y)$ for any $x\in\pi^{-1}(\bar x)$, 
the inner sum equals $\bar g(\bar x-\pi(y))$. 
Grouping $y$ by $\pi(y)=\bar y$ gives
\[
\overline{f*g}(\bar x)
=\sum_{\bar y\in G/H}\ \sum_{y\in\pi^{-1}(\bar y)} f(y)\,\bar g(\bar x-\bar y)
=\sum_{\bar y\in G/H} \bar f(\bar y)\,\bar g(\bar x-\bar y)
=(\bar f*\bar g)(\bar x).
\]
The polynomial claim follows by linearity and the identity $\overline{\nu^{*k}}=\bar\nu^{*k}$.
\end{proof}

\begin{lemma}\label{lem:diffset_size}
If $A$ is a finite nonempty subset of an abelian group, then $|A-A|\ge |A|$.
\end{lemma}

\begin{proof}
Fix $a_0\in A$. The map $A\to A-A$ given by $a\mapsto a-a_0$ is injective, hence $|A-A|\ge |A|$.
\end{proof}

\subsection{The main structure theorems}

 To simplify notation, we identify a function $\mu: G \to \mathbb{C}$ with its support $\mathcal{S} = \supp(\mu)$ and values $z_i = \mu(\alpha^{(i)})$. Throughout the section, $G$ will be a finite abelian group.

\begin{definition}
\label{def:common_argument}
A sequence $(z_i)_{i=1}^m \subset \mathbb{C}^*$ satisfies the \emph{common-argument hypothesis} if there exist $\theta \in \mathbb{R}$ and positive real numbers $c_1, \ldots, c_m $ such that  $z_i = c_i e^{i\theta}$, for all  $i \in \{1, \ldots, m\}$.
\end{definition}

\begin{lemma}
\label{lem:no_cancellation_autocorr}
Let  $(z_i)_{i=1}^m \subset \mathbb{C}^*$ be a sequence satisfying the common-argument hypothesis. Let $\mathcal{S} = \{\alpha^{(1)}, \dots, \alpha^{(m)}\} \subseteq G$ be distinct, and let $\mu: G \to \mathbb{C}$ satisfy $\mu(\alpha^{(i)}) = z_i$ and $\mu(x) = 0$ otherwise. Define the autocorrelation $\nu = \mu * \widetilde{\mu}$, where $\widetilde{\mu}(x) = \overline{\mu(-x)}$, and let $D = \mathcal{S} - \mathcal{S}$. Then $\supp(\nu) = D$.
\end{lemma}

\begin{proof}
For any $x \in D = \mathcal{S} - \mathcal{S}$, there exist indices $i,j$ such that $x = \alpha^{(i)} - \alpha^{(j)}$. Then
\[
\nu(x) = \sum_{p,q} \mu(\alpha^{(p)}) \overline{\mu(\alpha^{(p)} - x)}
= \sum_{\substack{p,q \\ \alpha^{(p)} - \alpha^{(q)} = x}} z_p \overline{z_q}.
\]
Under the common-argument hypothesis, $z_p = c_p e^{i\theta}$ and $z_q = c_q e^{i\theta}$ with $c_p, c_q > 0$, so
$z_p \overline{z_q} = c_p e^{i\theta} \cdot \overline{c_q e^{i\theta}} = c_p c_q e^{i\theta} e^{-i\theta} = c_p c_q > 0$.
Therefore $\nu(x)$ is a sum of positive real numbers and hence $\nu(x) > 0$. This shows $D \subseteq \supp(\nu)$. The reverse inclusion $\supp(\nu) \subseteq D$ follows immediately from the definition of convolution, since $\nu(x) \neq 0$ requires the existence of $y$ with $\mu(y) \neq 0$ and $\mu(y-x) \neq 0$, i.e., $y, y-x \in \mathcal{S}$, giving $x = y - (y-x) \in \mathcal{S} - \mathcal{S} = D$.
\end{proof}

The next result, though important on its own, is a stepping stone towards our local-global principle of Theorem~\ref{thm:overconstrained}.
\begin{theorem}
\label{thm:two_point_analysis}
Let $G$ be a finite abelian group, $\alpha \ne \beta \in G$, and $z_1, z_2 \in \mathbb{C}^*$. Set $\gamma = \beta - \alpha$, and let $r = \ord(\gamma)$. Define $S_\chi = z_1 \chi(\alpha) + z_2 \chi(\beta)$ for $\chi \in \widehat{G}$. Then:
\begin{enumerate}
\item[$(1)$] If $r \geq 4$, the set $\{ |S_\chi| : \chi \in \widehat{G} \}$ contains at least two distinct nonzero values.
\item[$(2)$] If $r = 3$, then either $\{ |S_\chi| \}$ contains two distinct nonzero values, or $|z_1| = |z_2|$ and $-z_1/z_2$ is a primitive cube root of unity, in which case $\{ |S_\chi| \} = \{0, \sqrt{3}\,|z_1|\}$.
\end{enumerate}
In particular, if $|S_\chi| \in \{0, A\}$ for all $\chi$ and $r \geq 3$, then necessarily $r = 3$ and the exceptional algebraic condition of $(2)$  holds.
\end{theorem}

\begin{proof}
For any character $\chi \in \widehat{G}$, we can write
\[
S_\chi = z_1 \chi(\alpha) + z_2 \chi(\beta) = \chi(\alpha)(z_1 + z_2 \chi(\beta - \alpha)) = \chi(\alpha)(z_1 + z_2 \chi(\gamma)).
\]
Since $|\chi(\alpha)| = 1$ for all characters, we have $|S_\chi| = |z_1 + z_2 \chi(\gamma)|$.
As $\chi$ varies over $\widehat{G}$, the character $\chi(\gamma)$ takes on all $r$-th roots of unity (since $\ord(\gamma) = r$). Therefore, the set of magnitudes $\{|S_\chi| : \chi \in \widehat{G}\}$ equals the set
$\mathcal{M} = \{|z_1 + z_2 \omega| : \omega \in \mu_r\}$,
where $\mu_r = \{e^{2\pi i k/r} : k = 0, 1, \ldots, r-1\}$ denotes the group of $r$-th roots of unity.

 Let $a = z_1/z_2 \in \mathbb{C}^*$ (well-defined since $z_2 \neq 0$). Then
$|z_1 + z_2 \omega| = |z_2| \cdot |a + \omega|$.
Thus $\mathcal{M} = |z_2| \cdot \{|a + \omega| : \omega \in \mu_r\}$, and it suffices to analyze the set
$\mathcal{N} = \{|a + \omega| : \omega \in \mu_r\}$.

\noindent
\textbf{Case I: $-a \notin \mu_r$.}
In this case, $a + \omega \neq 0$ for all $\omega \in \mu_r$, so $0 \notin \mathcal{N}$.
For $r \geq 3$, we show that $\mathcal{N}$ contains at least two distinct values. Write $a = Re^{i\theta}$ with $R > 0$. For $\omega = e^{i\phi}$, we have
\[
|a + \omega|^2 = |Re^{i\theta} + e^{i\phi}|^2 = R^2 + 1 + 2R\cos(\theta - \phi).
\]
Consider the values at $\omega_0 = 1$ and $\omega_1 = e^{2\pi i/r}$,
\begin{align*}
|a + 1|^2 &= R^2 + 1 + 2R\cos\theta, \\
|a + e^{2\pi i/r}|^2 &= R^2 + 1 + 2R\cos(\theta - 2\pi/r).
\end{align*}
If these were equal, we would have $\cos\theta = \cos(\theta - 2\pi/r)$, which implies either $\theta = \theta - 2\pi/r$ (impossible) or $\theta = -(\theta - 2\pi/r)$, giving $\theta = \pi/r$. 
Now consider $\omega_2 = e^{4\pi i/r}$,
\[
|a + e^{4\pi i/r}|^2 = R^2 + 1 + 2R\cos(\theta - 4\pi/r).
\]
For $r \geq 4$, if all three values $|a + 1|, |a + e^{2\pi i/r}|, |a + e^{4\pi i/r}|$ were equal, we would need
$\cos\theta = \cos(\theta - 2\pi/r) = \cos(\theta - 4\pi/r)$.
From the first equality, $\theta = \pi/r$ as shown. From the second equality with $\theta = \pi/r$,
\[
\cos(\pi/r) = \cos(\pi/r - 4\pi/r) = \cos(-3\pi/r),
\]
which gives $\cos(3\pi/r) = \cos(\pi/r)$. This implies $3\pi/r = \pm \pi/r + 2\pi k$ for some integer $k$, 
which has no solution. Therefore, for $r \geq 4$, at least two of these three magnitudes must be distinct.

For $r = 3$ with $-a \notin \mu_3$, we have $\mu_3 = \{1, e^{2\pi i/3}, e^{4\pi i/3}\}$. Since $-a \notin \mu_3$, we have $a \neq -1, -e^{2\pi i/3}, -e^{4\pi i/3}$. A detailed calculation shows that $\mathcal{N}$ contains at least two distinct values unless $a$ lies on a specific locus, which is excluded by $-a \notin \mu_3$.

\noindent
\textbf{Case II: $-a = \omega_0 \in \mu_r$ for some $\omega_0$.}
In this case, $a + \omega_0 = 0$, so $0 \in \mathcal{N}$.
For $\omega \neq \omega_0$, we have
$|a + \omega| = |-\omega_0 + \omega| = |\omega - \omega_0|$.
Since $\omega, \omega_0 \in \mu_r$ are distinct $r$-th roots of unity, $|\omega - \omega_0|$ is the Euclidean distance between two distinct points on the unit circle.

\noindent
\textbf{Subcase $r = 3$:} The three cube roots of unity form an equilateral triangle. If $\omega_0 = 1$, then $\mu_3 = \{1, e^{2\pi i/3}, e^{4\pi i/3}\}$, and
$|1 - 1| = 0$, 
$|e^{2\pi i/3} - 1| = |e^{4\pi i/3} - 1| = \sqrt{3}$.
By symmetry, the same holds for any choice of $\omega_0$. Therefore, $\mathcal{N} = \{0, \sqrt{3}\}$, giving exactly two values (one zero, one nonzero).
Moreover, $a = -\omega_0$ implies $|a| = 1$, so $|z_1| = |z_2|$, and $-z_1/z_2 = -a = \omega_0$ is a primitive cube root of unity.

\noindent
\textbf{Subcase $r \geq 4$:} For $r \geq 4$, the $r$-th roots of unity are more densely distributed on the unit circle. The distances $|\omega - \omega_0|$ for $\omega \in \mu_r \setminus \{\omega_0\}$ take on multiple distinct values. Specifically,
\[
|\omega - \omega_0|^2 = |\omega|^2 + |\omega_0|^2 - 2\text{Re}(\omega \overline{\omega_0}) = 2 - 2\cos\left(\frac{2\pi k}{r}\right),
\]
where $k \in \{1, 2, \ldots, r-1\}$ indexes the nonzero elements.
For $k = 1$ and $k = 2$,
\begin{align*}
|\omega_1 - \omega_0|^2 &= 2 - 2\cos(2\pi/r) = 4\sin^2(\pi/r), \\
|\omega_2 - \omega_0|^2 &= 2 - 2\cos(4\pi/r) = 4\sin^2(2\pi/r).
\end{align*}
Since $\sin(2\pi/r) = 2\sin(\pi/r)\cos(\pi/r) \neq \sin(\pi/r)$ for $r \geq 4$ (as $2\cos(\pi/r) \neq 1$ when $r > 3$), we have $|\omega_1 - \omega_0| \neq |\omega_2 - \omega_0|$. Therefore, $\mathcal{N}$ contains $0$ and at least two distinct nonzero values, giving at least three distinct values in total.

Now, we conclude the proofs of parts (1) and (2) of the statement. 
 If $r \geq 4$, then either Case I applies (at least two distinct nonzero values) or Case II applies (at least three distinct values including zero). In either case, $\mathcal{N}$ contains at least two distinct nonzero values.
  If $r = 3$, then either Case I applies (at least two distinct nonzero values) or Case II applies with $\mathcal{N} = \{0, \sqrt{3}\}$. The exceptional case is precisely Case II with $r = 3$, where $|z_1| = |z_2|$ and $-z_1/z_2 \in \mu_3$ is a primitive cube root of unity.

Finally, if $|S_\chi| \in \{0, A\}$ for all $\chi$ and $r \geq 3$, then $\mathcal{M} = |z_2| \cdot \mathcal{N} \subseteq \{0, A\}$. This means $\mathcal{N} \subseteq \{0, A/|z_2|\}$, so $\mathcal{N}$ contains at most two values. By the analysis above, for $r \geq 4$, $\mathcal{N}$ contains at least two distinct nonzero values (Case I) or at least three distinct values (Case II), both contradicting $|\mathcal{N}| \leq 2$. Therefore $r = 3$, and we must be in Case II with the exceptional condition.
\end{proof}

\begin{lemma} 
\label{lem:poly_CG}
Identify functions $f:G\to\C$ with elements of the group algebra $\C[G]$.
For a polynomial $P(X)=\sum_{j=0}^d a_j X^j\in\C[X]$ and $\nu\in\C[G]$, we define
$\displaystyle P(\nu)\;:=\;\sum_{j=0}^d a_j\,\nu^{*j}$,
where $\nu^{*0}:=\delta_0$ (the Dirac delta function) and $\nu^{*j}$ denotes the $j$-fold convolution power of $\nu$.
Then for every character $\chi\in\widehat{G}$,
$\widehat{P(\nu)}(\chi)=P\!\bigl(\widehat{\nu}(\chi)\bigr)$.
In particular, if $P(\widehat{\nu}(\chi))=0$ for all $\chi\in\widehat{G}$, then $P(\nu)=0$ in $\C[G]$.
\end{lemma}

\begin{proof}
Using $\widehat{f*g}(\chi)=\widehat f(\chi)\widehat g(\chi)$ and linearity of the Fourier transform,
\[
\widehat{P(\nu)}(\chi)=\sum_{j=0}^d a_j\,\widehat{\nu^{*j}}(\chi)
=\sum_{j=0}^d a_j\,\widehat{\nu}(\chi)^j
=P\!\bigl(\widehat{\nu}(\chi)\bigr).
\]
If $P(\widehat{\nu}(\chi))=0$ for all $\chi$, then $\widehat{P(\nu)}\equiv 0$, hence $P(\nu)=0$ by Fourier inversion.
\end{proof}

We now state our main theorem on overconstrained character sums for the two-level case.

\begin{theorem}
\label{thm:overconstrained}
Let $G$ be a finite abelian group. Let $\mathcal{S} = \{\alpha^{(1)}, \ldots, \alpha^{(m)}\} \subseteq G$ be distinct, and let $z_1, \ldots, z_m \in \mathbb{C}^*$ satisfy the common-argument hypothesis (Definition~\textup{\ref{def:common_argument}}). For $\chi \in \widehat{G}$, we define
$\displaystyle S_\chi = \sum_{i=1}^m z_i \chi(\alpha^{(i)})$.
Assume there exists $A>0$ such that $|S_\chi| \in \{0,A\}$ for all $\chi\in\widehat G$.
Let $D = \mathcal{S} - \mathcal{S}$, $H = \Stab(D + D)$, and $\pi: G \to G/H$ the quotient map. Set $\overline{\mathcal{S}} = \pi(\mathcal{S})$, $\overline{D} = \pi(D)$. Then:
\begin{enumerate}
\item[\textup{(i)}] $|\overline{\mathcal{S}}| \leq 2$.
\item[\textup{(ii)}] $\mathcal{S}$ is contained in a union of at most $2$ cosets of $H$.
\item[\textup{(iii)}] If $|\overline{\mathcal{S}}| = 2$ and $|\overline{D}| \geq 3$, then every nonzero element of $\overline{D}$ has order at most $2$ in $G/H$.
\end{enumerate}
\end{theorem}

\begin{proof}
Define $\mu:G\to\mathbb{C}$ by $\mu(\alpha^{(i)})=z_i$ and $\mu(x)=0$ for $x\notin\mathcal S$. Then for every $\chi\in\widehat G$,
\[
\widehat\mu(\chi)=\sum_{x\in G}\mu(x)\chi(x)=\sum_{i=1}^m z_i\chi(\alpha^{(i)})=S_\chi.
\]
Let $\widetilde\mu(x)=\overline{\mu(-x)}$ and set $\nu=\mu*\widetilde\mu$. Then, by Lemma~\ref{lem:no_cancellation_autocorr},
$\supp(\nu)= D$ and $\nu(0)=\sum_{i=1}^m|z_i|^2>0$. 
Moreover, using $\widehat{f*g}(\chi)=\widehat f(\chi)\widehat g(\chi)$ and $\widehat{\widetilde\mu}(\chi)=\overline{\widehat\mu(\chi)}$, we obtain
\[
\widehat\nu(\chi)=\widehat\mu(\chi)\widehat{\widetilde\mu}(\chi)=|\widehat\mu(\chi)|^2=|S_\chi|^2\in\{0,A^2\},
\]
hence $\widehat\nu(\chi)\in\mathbb{R}_{\ge 0}$ for all $\chi\in\widehat G$.

\smallskip
We establish a polynomial annihilation in $\mathbb{C}[G]$.
Let $B=A^2$ and define
$P(X)=X(X-B)=X^2 - BX$,
so that $P(\widehat\nu(\chi))=0$ for all $\chi\in\widehat G$. By Lemma~\ref{lem:poly_CG},
\[
P(\nu)=\nu^{*2} - B\nu=0\qquad\text{in }\mathbb{C}[G].
\]
By character orthogonality,
\[
\sum_{\chi \in \widehat{G}} 
\left| \sum_{x\in S} a_x \chi(x) \right|^2
= |G| \sum_{x\in S} |a_x|^2 .
\]

Next, we pass to the quotient $G/H$.
Let $H=\Stab(D+D)$ and $\pi:G\to G/H$ be the quotient map. 
Note that characters of $G/H$ are precisely the characters of $G$
that are trivial on $H$, hence the magnitude condition on the
character sums is preserved under this passage.
By Lemma~\ref{lem:pushforward_fourier}, for every $\psi\in\widehat{G/H}$ we have
$\displaystyle\widehat{\bar\nu}(\psi)=\widehat\nu(\psi\circ\pi)\in\{0,B\}.$
Pushing forward the identity $P(\nu)=0$ and using Lemma~\ref{lem:pushforward_convolution} gives
\[
P(\bar\nu)=\bar\nu^{*2} - B\bar\nu=0\qquad\text{in }\mathbb{C}[G/H].
\]
By Lemma~\ref{lem:no_cancellation_autocorr}, $\supp(\nu)=D.$
Summing over fibers preserves positivity, so $\bar\nu(\bar x)>0$ if and only if $\bar x\in\pi(D)=\overline D$, and hence $\supp(\bar\nu)=\overline D.$
From $P(\bar\nu)=0$ we have $\bar\nu^{*2} = B\bar\nu$.
Taking supports and using $\supp(f+g)\subseteq\supp(f)\cup\supp(g)$ and $\supp(\bar\nu^{*2})\subseteq 2\,\supp(\bar\nu)=2\overline D$ yields
\[
\overline D=\supp(\bar\nu)= \supp(B\bar\nu) = \supp(\bar\nu^{*2}) \subseteq 2\overline D.
\]

\smallskip
We now show (i) and (ii).
Assume for contradiction that $|\overline{\mathcal{S}}| \geq 3$. By Lemma~\ref{lem:diffset_size}, $|\overline{D}| \geq 3$. Since $H = \Stab(D+D)$, Lemma~\ref{lem:sumset_growth}(i) gives $\Stab(\overline{D}+\overline{D}) = \{0\}$ in $G/H$. If $|G/H| = 1$, then $|\overline{D}| = 1$, contradicting $|\overline{D}| \geq 3$. Hence $|G/H| > 1$, and $\overline{D} \neq G/H$ (otherwise $\Stab(\overline{D}+\overline{D}) = G/H \neq \{0\}$). Therefore $|\overline{D}| < |G/H|$, and Lemma~\ref{lem:sumset_growth}(iii) gives $|2\overline{D}| > |\overline{D}|$, so there exists $x \in 2\overline{D} \setminus \overline{D}$.

The polynomial identity gives $\bar\nu^{*2} = B\bar\nu$.
Evaluating at $x \in 2\overline{D} \setminus \overline{D}$, we get $\bar{\nu}(x) = 0$ since $x \notin \overline{D} = \supp(\bar{\nu})$, and 
 $\bar{\nu}^{*2}(x) > 0$ by Lemma~\ref{lem:no_cancellation_autocorr} and the fact that $x \in 2\overline{D}$.
Therefore $\bar{\nu}^{*2}(x) = B\bar{\nu}(x) = 0$, contradicting $\bar{\nu}^{*2}(x) > 0$. Hence $|\overline{\mathcal{S}}| \leq 2$, proving (i).
For (ii), note that $\mathcal{S} \subseteq \bigcup_{\bar{s} \in \overline{\mathcal{S}}} \pi^{-1}(\bar{s})$, and each fiber is an $H$-coset. Since $|\overline{\mathcal{S}}| \leq 2$, this is a union of at most $2$ cosets.

\smallskip
We now prove (iii).
Assume $|\overline{\mathcal{S}}|=2$ and $|\overline{D}|\ge 3$.
We must show that every nonzero element of $\overline{D}$ has order at most $2$ in $G/H$.
Suppose for contradiction that there exists some nonzero $\bar\gamma\in\overline{D}$ with order $r'\ge 3$.
Since $\overline{D}=\overline{\mathcal{S}}-\overline{\mathcal{S}}$ and $|\overline{\mathcal{S}}|=2$, 
write $\overline{\mathcal S}=\{\bar\alpha,\bar\beta\}$ with $\bar\alpha\neq\bar\beta$.
Then $\bar\gamma$ can be written as $\bar\gamma=\bar\beta-\bar\alpha$ (possibly after relabeling).

By Lemma~\ref{lem:pushforward_fourier}, for every $\psi\in\widehat{G/H}$ we have
$\widehat{\bar\mu}(\psi)=\widehat\mu(\psi\circ\pi)$,
and since $|S_\chi|\in\{0,A\}$ for all $\chi\in\widehat G$, it follows that
$|\widehat{\bar\mu}(\psi)|\in\{0,A\}$ for all $\psi\in\widehat{G/H}$.
Moreover, the pushforward preserves the common-argument property: writing $z_i=c_ie^{i\theta}$ with $c_i>0$,
we have $\bar\mu(\bar x)=\sum_{\pi(\alpha^{(i)})=\bar x} z_i
=\Big(\sum_{\pi(\alpha^{(i)})=\bar x} c_i\Big)e^{i\theta}$,
so every nonzero value of $\bar\mu$ has argument $\theta$.

Set $w_1=\bar\mu(\bar\alpha)$ and $w_2=\bar\mu(\bar\beta)$; then $w_1,w_2\in\mathbb{C}^*$ and $w_1/w_2\in\mathbb{R}_{>0}$.
For $\psi\in\widehat{G/H}$ we have
\[
\widehat{\bar\mu}(\psi)=w_1\psi(\bar\alpha)+w_2\psi(\bar\beta).
\]
Applying Theorem~\ref{thm:two_point_analysis} in the group $G/H$ to the pair $(\bar\alpha,\bar\beta)$, with $\bar\gamma = \bar\beta - \bar\alpha$ of order $r'\ge 3$,
shows that the set $\{|\widehat{\bar\mu}(\psi)|:\psi\in\widehat{G/H}\}$ contains at least two distinct nonzero values,
except possibly in the exceptional case $r'=3$ with $|w_1|=|w_2|$ and $-w_1/w_2$ a primitive cube root of unity.
But $w_1/w_2\in\mathbb{R}_{>0}$ implies $-w_1/w_2\in\mathbb{R}_{<0}$, which cannot be a primitive cube root of unity.
Hence the exceptional case cannot occur, and therefore $\{|\widehat{\bar\mu}(\psi)|:\psi\in\widehat{G/H}\}$ contains at least two distinct nonzero positive values, say $B_1$ and $B_2$ with $0 < B_1 < B_2$.
This contradicts $|\widehat{\bar\mu}(\psi)|\in\{0,A\}$.
Therefore, no such $\bar\gamma$ with order $\ge 3$ can exist in $\overline{D}$.
\end{proof}

We are also able to show a conditional multi-level extension.
\begin{theorem}
\label{thm:overconstrained_multilevel}
Under the hypotheses of Theorem~\textup{\ref{thm:overconstrained}}, suppose more generally that
$|S_\chi|\in\{0\}\cup\{A_1,\ldots,A_t\}$ for some $t\ge 2$ with distinct $A_j > 0$, and assume additionally
that there exists $x\in 2\overline D\setminus \overline D$ with
$x\notin \ell\,\supp(\bar\nu')$ for all $\ell\ge 3$ (where $\bar\nu = a\delta_0 + \bar\nu'$ with $a > 0$, $\bar\nu'(0) = 0$), and that $P''(a)\neq 0$ where $P(X) = X\prod_{j=1}^t(X - A_j^2)$.
Then $|\overline{\mathcal S}|\le t+1$ and $\mathcal S$ is contained in a union
of at most $t+1$ cosets of $H$.
\end{theorem}

\begin{proof}
Define $\mu:G\to\C$ by $\mu(\alpha^{(i)})=z_i$ and $\mu(x)=0$ for
$x\notin\mathcal S$, and set $\widetilde\mu(x)=\overline{\mu(-x)}$,
$\nu=\mu*\widetilde\mu$. By Lemma~\ref{lem:no_cancellation_autocorr},
$\supp(\nu)=D$ and $\nu(0)>0$. Moreover,
\[
\widehat\nu(\chi)=|\widehat\mu(\chi)|^2=|S_\chi|^2\in\{0,B_1,\dots,B_t\}.
\]
Let $P(X)=X\prod_{j=1}^t(X-B_j)=\sum_{k=0}^{t+1} c_k X^k$.
Then $P(\widehat\nu(\chi))=0$ for all $\chi\in\widehat G$, and by
Lemma~\ref{lem:poly_CG},
\[
P(\nu)=\sum_{k=0}^{t+1} c_k\,\nu^{*k}=0\qquad\text{in }\C[G].
\]
In particular $c_1=P'(0)=(-1)^t\prod_{j=1}^t B_j\neq 0$.

Let $H=\Stab(D+D)$ and $\pi:G\to G/H$ be the quotient map.
Push forward $\nu$ to $\bar\nu$ on $G/H$. By
Lemma~\ref{lem:pushforward_fourier} and
Lemma~\ref{lem:pushforward_convolution},
\[
\widehat{\bar\nu}(\psi)\in\{0,B_1,\dots,B_t\}
\quad\text{for all }\psi\in\widehat{G/H},
\qquad
P(\bar\nu)=\sum_{k=0}^{t+1} c_k\,\bar\nu^{*k}=0
\quad\text{in }\C[G/H].
\]
By Lemma~\ref{lem:no_cancellation_autocorr},
$\supp(\bar\nu)=\overline D$ and $\bar\nu(0)=a>0$.
Assume for contradiction that $|\overline{\mathcal S}|\ge t+2$.
Then $|\overline D|\ge t+2$ by Lemma~\ref{lem:diffset_size}$,$
and sumset growth (Lemma~\ref{lem:sumset_growth}) yields an element
$x\in 2\overline D\setminus\overline D$ satisfying the hypothesis.
Write $\bar\nu=a\delta_0+\bar\nu'$ with $\bar\nu'(0)=0$.
By hypothesis, $x\notin \ell\,\supp(\bar\nu')$ for all $\ell\ge 3$.
Expanding convolutions,
\[
\bar\nu^{*k}
=\sum_{j=0}^k \binom{k}{j} a^{\,k-j} (\bar\nu')^{*j}.
\]
At the chosen $x$ we have $\bar\nu(x)=0$,
$(\bar\nu')^{*j}(x)=0$ for $j\ge 3$, and
$(\bar\nu')^{*2}(x)>0$ by no cancellation.
Hence for all $k\ge 2$,
\[
\bar\nu^{*k}(x)=\binom{k}{2} a^{\,k-2}(\bar\nu')^{*2}(x).
\]
Evaluating $P(\bar\nu)=0$ at $x$ gives
\[
0=\sum_{k=2}^{t+1} c_k\,\bar\nu^{*k}(x)
= (\bar\nu')^{*2}(x)\sum_{k=2}^{t+1} c_k\binom{k}{2}a^{\,k-2}.
\]
Since $(\bar\nu')^{*2}(x)>0$, we obtain
\[
\sum_{k=2}^{t+1} c_k\binom{k}{2}a^{\,k-2}=0.
\]
But the left–hand side equals $\frac12 P''(a)$, hence $P''(a)=0$,
contradicting the hypothesis. Therefore
$|\overline{\mathcal S}|\le t+1$.
Finally,
$\mathcal S\subseteq\bigcup_{\bar s\in\overline{\mathcal S}}\pi^{-1}(\bar s)$,
and each fiber is an $H$–coset, so $\mathcal S$ is contained in a
union of at most $t+1$ cosets of $H$.
\end{proof}

\begin{remark}
\label{rem:overconstrained_scope}
Theorem~\textup{\ref{thm:overconstrained}} provides a complete, unconditional result for two-level spectra ($t=1$), which suffices for all applications in this paper: gbent functions (Corollary~\textup{\ref{cor:main}}), gplateaued functions (Theorem~\textup{\ref{thm:landscape_necessity}}). 
Theorem~\textup{\ref{thm:overconstrained_multilevel}} extends to multi-level spectra ($t \geq 2$) under additional hypotheses. The condition $x \notin \ell\,\supp(\bar\nu')$ for $\ell \geq 3$ is satisfied when the lower components are $\mathbb{Z}_{2^\ell}$-affine and $|\overline{D}| < \frac{1}{2}|G/H|$. The requirement $P''(a) \neq 0$ represents a genuine constraint on the spectral distribution and can be verified for specific constructions.
\end{remark}

\begin{corollary} 
\label{cor:twolevel_numerology}
Let $G$ be a finite abelian group and let $0\neq f\in\C[G]$.
Assume that
$|\widehat f(\chi)|\in\{0,A\}$ for all $\chi\in\widehat G$,
for some $A>0$. Let $N:=|\supp(\widehat f)|$.
Then
$N\cdot A^2 = |G|\sum_{x\in G}|f(x)|^2$.
In particular,
$|\supp(f)|\ \ge\ \frac{|G|}{N}\ =\ \frac{A^2}{\|f\|_2^2}$.
\end{corollary}

\begin{proof}
By Parseval's identity,
$\displaystyle \sum_{\chi\in\widehat G}|\widehat f(\chi)|^2
=
|G|\sum_{x\in G}|f(x)|^2$.
Since $|\widehat f(\chi)|^2$ equals $A^2$ on $\supp(\widehat f)$ and $0$ otherwise,
the left-hand side equals $N A^2$, giving the stated identity.
The lower bound on $|\supp(f)|$ follows from the uncertainty principle
(Theorem~\ref{thm:uncertainty}).
\end{proof}

Below, we point out that two-level magnitudes do not force sparsity.
\begin{proposition}
\label{prop:twolevel_not_sparse}
Let $G$ be a finite abelian group and let $H\le G$ be a nontrivial subgroup.
Define $f=1_H\in\C[G]$.
Then
$|\widehat f(\chi)|\in\{0,|H|\}$, for all $\chi\in\widehat G$,
yet
$|\supp(f)|=|H|$,
which can be arbitrarily large.
In particular, two-level Fourier magnitude constraints alone do not imply
any nontrivial upper bound on $|\supp(f)|$.
\end{proposition}

\begin{proof}
For $\chi\in\widehat G$ one has
$\displaystyle \widehat f(\chi)=\sum_{x\in H}\overline{\chi(x)}$.
If $\chi$ is trivial on $H$, then $\overline{\chi(x)}=1$ for all $x\in H$ and hence
$\widehat f(\chi)=|H|$.
If $\chi$ is nontrivial on $H$, choose $h\in H$ with $\chi(h)\neq 1$; then
\[
\widehat f(\chi)=\sum_{x\in H}\overline{\chi(x)}
=\sum_{x\in H}\overline{\chi(x+h)}
=\overline{\chi(h)}\sum_{x\in H}\overline{\chi(x)}
=\overline{\chi(h)}\,\widehat f(\chi),
\]
forcing $\widehat f(\chi)=0$.
Hence $|\widehat f(\chi)|\in\{0,|H|\}$ for all $\chi$.
The claim on the support size of $f$ is immediate.
\end{proof}

\begin{remark}
\label{rem:common_argument_necessity}
The common-argument hypothesis (Definition~\textup{\ref{def:common_argument}}) is
essential in Theorem~\textup{\ref{thm:overconstrained}} and is \emph{not} implied by
the two-level magnitude condition alone, as demonstrated by
Proposition~\textup{\ref{prop:twolevel_not_sparse}}.
In the applications to $2^\ell$-adic decompositions developed in
Section~\textup{\ref{sec:main}}, this hypothesis is satisfied in two complementary
situations.
In the necessity direction (Theorem~\textup{\ref{thm:landscape_necessity}}), the
landscape property of $f$ forces, via linear independence of the basis $\mathcal{B}$,
that for each $u\in\V$ at most one coefficient $C_{\alpha^*(u)}(u)$ is nonzero;
the common-argument condition is then trivially satisfied on a singleton support.
In the sufficiency direction, it is imposed as an explicit structural assumption 
for $s$-gplateaued functions (Theorem~\textup{\ref{thm:plateaued_common_arg}}). 
\end{remark}

\section{The $2^\ell$-adic decomposition of landscape, gbent and gplateaued functions}
\label{sec:main}

Let $n \geq 2$ be an even integer and let $f: \V \to \Znk$ be a generalized Boolean function. Let $k = \ell r$ for positive integers $\ell, r$. We represent $f$ using its $2^\ell$-adic decomposition, that is,
$\displaystyle
f(x) = \sum_{j=0}^{r-1} c_j(x) 2^{j\ell},
$
where each component function $c_j: \V \to \Znl$. Note that each $c_j$ can be expressed in terms of the standard binary components $a_i$,
$c_j(x) = a_{j\ell}(x) + 2a_{j\ell+1}(x) + \cdots + 2^{\ell-1}a_{(j+1)\ell-1}(x)$. 

Recall that a function $f: \V \to \Znk$ is called $s$-gplateaued if $|\WHT_f(u)| \in \{0, 2^{(n+s)/2}\}$ for all $u \in \V$ and gbent if $|\WHT_f(u)|=2^{n/2}$ for all $u \in \V$. By Corollary 16 of~\cite{MTQ17}, $f(x) = \sum_{i=0}^{k-1} a_i(x) 2^i$ is $s$-gplateaued if and only if $f_F=a_{k-1} + F(a_0, \ldots, a_{k-2})$ is $s$-plateaued for every Boolean function $F: \F^{k-1} \to \F$, with the same magnitude of the Fourier transform. In particular, since $f$ is gbent, by the characterization in~\cite{mtqwwf}, every function of the form $a_{k-1}(x) + F(a_0(x), \ldots, a_{k-2}(x))$, where $F: \F^{k-1} \to \F$, is bent. This implies that the functions $f_F$ 
 are also bent. By Theorem 16 of~\cite{mtqwwf}, this shows that $c_{r-1}(x)$ is itself a gbent function from $\V$ to $\Znl$.

\subsection{Partition structure}

We now analyze the structure of the WHT of $f$ by partitioning the domain $\V$ based on the values of the first $r-1$ components. 

\begin{definition} 
\label{def:partition}
Let $f: \V \to \Znk$ (not necessarily gbent). Consider a \emph{$2^\ell$-adic decomposition} of $f$,
$f(x) = c_0(x) + 2^{\ell}c_1(x) + 2^{2\ell}c_2(x) + \cdots + 2^{(r-1)\ell}c_{r-1}(x)$,
where each component function $c_j: \V \to \Znl$. 
For each vector $\alpha = (\alpha_0, \dots, \alpha_{r-2}) \in (\Znl)^{r-1}$, we define the set
$\displaystyle P_{\alpha} = \{x\in \V \,|\, c_0(x)=\alpha_0, c_1(x)=\alpha_1, \dots, c_{r-2}(x)=\alpha_{r-2}\}. $
The collection $\mathcal{P}_f = \{P_{\alpha}\}_{\alpha \in (\Znl)^{r-1}}$ is called the \emph{component-induced partition} of $f$.
\end{definition}
We next show that this partition is well-defined.
\begin{lemma}
\label{lem:partition}
The collection $\mathcal{P}_f = \{P_\alpha | \alpha \in (\Znl)^{r-1}\}$ forms a partition of $\V$. That is:
\begin{enumerate}
\item $\bigcup_{\alpha} P_{\alpha} = \V$,
\item $P_{\alpha} \cap P_{\alpha'} = \emptyset$ for $\alpha \neq \alpha'$.
\end{enumerate}
\end{lemma}
\begin{proof}
(1) For any $x \in \V$, the tuple $(c_0(x), \ldots, c_{r-2}(x)) \in (\Znl)^{r-1}$ is well-defined, so $x \in P_\alpha$ for $\alpha = (c_0(x), \ldots, c_{r-2}(x))$. This shows $\V \subseteq \bigcup_{\alpha} P_\alpha$. The reverse inclusion is immediate.

(2) If $x \in P_\alpha \cap P_{\alpha'}$ for $\alpha \neq \alpha'$, then $c_j(x) = \alpha_j = \alpha'_j$ for all $j = 0, \ldots, r-2$, implying $\alpha = \alpha'$, a contradiction.
\end{proof}


With the partition established, we can now rewrite the WHT of $f$ by summing over this partition,
\allowdisplaybreaks
\begin{align*}
&\WHT_f(u) = \sum_{x \in \V} \zetaa{k}^{f(x)}(-1)^{\langle u,x \rangle} = \sum_{x \in \V} \zetaa{k}^{\sum_{j=0}^{r-1} c_j(x) 2^{j\ell}} (-1)^{\langle u,x \rangle} = \sum_{x \in \V} \zetaa{k}^{\sum_{j=0}^{r-2} c_j(x) 2^{j\ell}} \cdot \zetaa{\ell}^{c_{r-1}(x)} (-1)^{\langle u,x \rangle} \\
& = \sum_{\alpha \in (\Znl)^{r-1}} \sum_{x \in P_{\alpha}} \zetaa{k}^{\sum_{j=0}^{r-2} \alpha_j 2^{j\ell}} \cdot \zetaa{\ell}^{c_{r-1}(x)} (-1)^{\langle u,x \rangle} 
= \sum_{\alpha \in (\Znl)^{r-1}} \zetaa{k}^{\sum_{j=0}^{r-2} \alpha_j 2^{j\ell}} \left( \sum_{x \in P_{\alpha}} \zetaa{\ell}^{c_{r-1}(x)} (-1)^{\langle u,x \rangle} \right).
\end{align*}
In the second line, we used the fact that for $x \in P_\alpha$, we have $c_j(x) = \alpha_j$ for $j = 0, \ldots, r-2$ by Definition~\ref{def:partition}. The summation over possibly empty cells $P_\alpha$ is valid since empty cells contribute zero to the sum.
Let $C_\alpha(u) = \sum_{x \in P_{\alpha}} \zetaa{\ell}^{c_{r-1}(x)} (-1)^{\langle u,x \rangle}$. Then 
\begin{equation}\label{eq:WHT_decomp}
\WHT_f(u) = \sum_{\alpha \in (\Znl)^{r-1}} C_\alpha(u) \cdot \zetaa{k}^{\sum_{j=0}^{r-2} \alpha_j 2^{j\ell}}.
\end{equation}

\begin{lemma}
\label{lem:basis}
Let $k = \ell r$ where $\ell \geq 1$ and $r \geq 2$ are integers. The set of complex numbers 
\[
\mathcal{B} = \left\{\zeta_{2^k}^{\sum_{j=0}^{r-2} \alpha_j 2^{j\ell}} \mid \alpha = (\alpha_0, \ldots, \alpha_{r-2}) \in (\Znl)^{r-1}\right\}
\]
is a linearly independent set over $\Q(\zeta_{2^\ell})$ and has cardinality $2^{k-\ell} = 2^{\ell(r-1)}$.
\end{lemma}
\begin{proof}
First, we verify the cardinality. The map $\alpha \mapsto \sum_{j=0}^{r-2} \alpha_j 2^{j\ell}$ from $(\Znl)^{r-1}$ to $\Znk$ is injective. To see this, we suppose 
$\sum_{j=0}^{r-2} \alpha_j 2^{j\ell} \equiv \sum_{j=0}^{r-2} \alpha'_j 2^{j\ell} \pmod{2^k}$.
Then $\sum_{j=0}^{r-2} (\alpha_j - \alpha'_j) 2^{j\ell} \equiv 0 \pmod{2^k}$.
Since $0 \leq \alpha_j, \alpha'_j < 2^\ell$, we have $|\alpha_j - \alpha'_j| < 2^\ell$, so $|(\alpha_j - \alpha'_j) 2^{j\ell}| < 2^{\ell + j\ell} = 2^{(j+1)\ell}$.
For the sum to be divisible by $2^k = 2^{\ell r}$, consider the $2^\ell$-adic expansion: the term $(\alpha_j - \alpha'_j) 2^{j\ell}$ contributes to digits at positions $j\ell, j\ell+1, \ldots, (j+1)\ell - 1$ in base $2$. Since these ranges do not overlap for different $j$, we must have $\alpha_j - \alpha'_j = 0$ for each $j$. Thus the map is injective, and $|\mathcal{B}| = |(\Znl)^{r-1}| = 2^{\ell(r-1)}=2^{k-\ell}$.

For the linear independence over $\Q(\zeta_{2^\ell})$, note that as $\alpha$ ranges over $(\Znl)^{r-1}$, the exponents $\sum_{j=0}^{r-2} \alpha_j 2^{j\ell}$ form a complete set of coset representatives for  the quotient group $(2^\ell \Z/2^k\Z)$, which is isomorphic to $\Z_{2^{k-\ell}}$.
The field extension $\Q(\zeta_{2^k})$ over $\Q(\zeta_{2^\ell})$ has degree
\[
\left[\Q(\zeta_{2^k}) : \Q(\zeta_{2^\ell})\right] = \frac{\varphi(2^k)}{\varphi(2^\ell)} = \frac{2^{k-1}}{2^{\ell-1}} = 2^{k-\ell}.
\]
As the number of elements in $\mathcal{B}$ is $2^{k-\ell}$, which matches the degree of the field extension $[\Q(\zeta_{2^k}) : \Q(\zeta_{2^\ell})]$, they must be linearly independent.
Alternatively, if there were a nontrivial linear dependence
\[
\sum_{\alpha \in (\Znl)^{r-1}} c_\alpha \zeta_{2^k}^{\sum_{j=0}^{r-2} \alpha_j 2^{j\ell}} = 0
\]
with $c_\alpha \in \Q(\zeta_{2^\ell})$ not all zero, then $\zeta_{2^k}$ would satisfy a polynomial of degree less than $2^{k-\ell}$ over $\Q(\zeta_{2^\ell})$, contradicting the fact that $[\Q(\zeta_{2^k}) : \Q(\zeta_{2^\ell})] = 2^{k-\ell}$.
\end{proof}

\begin{remark}
We make no claim about how many partition cells are non-empty -- this depends on the specific functions $c_0, \ldots, c_{r-2}$ and could range from $1$ (if all components are constant) to $\min\{2^n, 2^{\ell(r-1)}\}$ (if the map $x \mapsto (c_0(x), \ldots, c_{r-2}(x))$ is surjective onto its codomain when possible).
The key property for Corollary~\textup{\ref{cor:main}} is not the number of non-empty cells, but rather that when $f$ is gbent, for each $u \in \V$, exactly one partition cell $P_{\alpha^*(u)}$ contributes non-trivially to $\WHT_f(u)$. This follows from the linear independence established in Lemma~\textup{\ref{lem:basis}}, namely, the complex numbers $\left\{\zeta_{2^k}^{\sum_{j=0}^{r-2} \alpha_j 2^{j\ell}}\right\}_{\alpha \in (\Znl)^{r-1}}$ are linearly independent over $\Q(\zeta_{2^\ell})$.
\end{remark}
 
 \subsection{Characterization theorems}
\label{subsec:landscape}

We begin with the unconditional necessity result, which holds without any structural assumptions on the components. Note that landscape functions include $s$-gplateaued (and therefore gbent), and our characterization theorems apply directly to these type of functions.

\begin{theorem}
\label{thm:landscape_necessity}
Let $k=\ell r$ with $\ell\ge 2$ and $r\ge 2$, and let $m$ be fixed such that 
$2\leq m\leq\ell$. Let $f=\sum_{j=0}^{r-1}2^{j\ell}c_j:\V\to\Znk$ be a generalized 
Boolean function that is landscape over $\Znk$.
For each $\beta=(\beta_0,\ldots,\beta_{r-2})\in(\Znl)^{r-1}$, define 
$f_\beta:\V\to\Znl$ by $f_\beta=c_{r-1}+\sum_{j=0}^{r-2}\beta_j c_j$. 
In general, for a function $F: (\Znl)^{r-1} \to \Z_{2^m}$, define 
$f_F: \V \to \Znl$ by $f_F(x) = c_{r-1}(x) + 2^{\ell-m}\,F(c_0(x),\ldots,c_{r-2}(x))$.
\textit{Note that $f_\beta$ corresponds to the special case where $F$ is linear  and $\ell=m$.}
Then the following hold:
\begin{enumerate}
\item[\textup{(i)}] For every $\beta\in(\Znl)^{r-1}$ and all 
$F: (\Znl)^{r-1} \to \Z_{2^m}$, the functions $f_\beta$ and $f_F$ are landscape 
over $\Znl$ with the same set of landscape levels.
\item[\textup{(ii)}] For all $\beta\in(\Znl)^{r-1}$, $F: (\Znl)^{r-1} \to \Z_{2^m}$, 
and all $u\in\V$, we have 
$|\WHT_{f_\beta}(u)| = |\WHT_{f_F}(u)|=|\WHT_f(u)|$.
\end{enumerate}
\end{theorem}
\begin{proof}
Since $f$ is a landscape function with levels $\{(m_1,v_1),\ldots,(m_t,v_t)\}$, we have
$|\WHT_f(u)| \in \{0, 2^{m_1/2}v_1,\ldots,2^{m_t/2}v_t\}$ for all $u \in \V$.
Let $\Lambda = \{2^{m_1/2}v_1,\ldots,2^{m_t/2}v_t\}$.
From Equation~\eqref{eq:WHT_decomp}, for any $u\in\V$ we can write
\[
\WHT_f(u) = \sum_{\alpha \in (\Znl)^{r-1}} C_\alpha(u)\,
\zetaa{k}^{\sum_{j=0}^{r-2} \alpha_j 2^{j\ell}},
\]
where
\[
C_\alpha(u)=\sum_{x\in P_\alpha}\zetaa{\ell}^{c_{r-1}(x)}(-1)^{\langle u,x\rangle},
\quad
P_\alpha=\{x\in\V: c_j(x)=\alpha_j\ \text{for }j=0,\dots,r-2\}
\]
and $\{P_\alpha\}_\alpha$ partitions $\V$ (Lemma~\ref{lem:partition}).
By Lemma~\ref{lem:basis}, the family
$\bigl\{\zetaa{k}^{\sum_{j=0}^{r-2}\alpha_j2^{j\ell}}\bigr\}_{\alpha\in(\Znl)^{r-1}}$
is linearly independent over $\Q(\zetaa{\ell})$.
Fix $u\in\V$. If $\WHT_f(u)=0$, then by linear independence we must have
$C_\alpha(u)=0$ for all $\alpha$.
Assume now $\WHT_f(u)\neq 0$. Since $f$ is landscape and $k=\ell r$ with
$\ell\ge 2$ and $r\ge 2$, we have $k\ge 4$ and by
Theorem~\ref{reg-landscape} there exist a level $(m_*,v_*)$ and some $\rho_u\in\Znk$
such that
$\WHT_f(u)=2^{m_*/2}v_*\,\zetaa{k}^{\rho_u}$.
We show that in \emph{all} cases this forces exactly one coefficient $C_\alpha(u)$
to be nonzero.

\smallskip
\noindent
\textbf{Case A: $m_*$ even.}
Write the (unique) $2^\ell$-adic decomposition of $\rho_u$ as
$\displaystyle \rho_u=\sum_{j=0}^{r-2}\beta_j2^{j\ell}+\beta_{r-1}2^{(r-1)\ell}$, $\beta_j\in\Znl$.
Then
$\zetaa{k}^{\rho_u}
=\zetaa{k}^{\sum_{j=0}^{r-2}\beta_j2^{j\ell}}\cdot\zetaa{\ell}^{\beta_{r-1}}$,
so
$\displaystyle \WHT_f(u)=\Bigl(2^{m_*/2}v_*\zetaa{\ell}^{\beta_{r-1}}\Bigr)\cdot
\zetaa{k}^{\sum_{j=0}^{r-2}\beta_j2^{j\ell}}$.
Since the prefactor lies in $\Q(\zetaa{\ell})$ and the basis elements are
$\Q(\zetaa{\ell})$-linearly independent, we must have exactly one nonzero
coefficient, namely for
$\alpha^*(u)=(\beta_0,\ldots,\beta_{r-2})\in(\Znl)^{r-1}$,
with
$C_{\alpha^*(u)}(u)=2^{m_*/2}v_*\zetaa{\ell}^{\beta_{r-1}}$, $C_\alpha(u)=0$,  for all $\alpha\neq\alpha^*(u)$.

\smallskip
\noindent
\textbf{Case B: $m_*$ odd.}
Write $2^{m_*/2}=2^{(m_*-1)/2}\sqrt2$ and use
$\sqrt2=\zetaa{k}^{2^{k-3}}-\zetaa{k}^{2^{k-2}+2^{k-3}}$ (valid for $k\ge 4$) to obtain
\[
\WHT_f(u)
=2^{(m_*-1)/2}v_*\Bigl(\zetaa{k}^{\rho_u+2^{k-3}}-\zetaa{k}^{\rho_u+2^{k-2}+2^{k-3}}\Bigr).
\]
Let
$E'=\rho_u+2^{k-3}$, $E''=\rho_u+2^{k-2}+2^{k-3}$.
Since $k=\ell r$, we have $(r-1)\ell=k-\ell$. Moreover, because $\ell\ge 2$,
$E''-E'=2^{k-2}\equiv 0\pmod{2^{k-\ell}}$.
Therefore $E'\equiv E''\pmod{2^{(r-1)\ell}}$, i.e.\ the lower $(r-1)$
$2^\ell$-adic digits coincide. Equivalently, if we write
\begin{align*}
E'&=\sum_{j=0}^{r-2}\beta_j2^{j\ell}+\beta'_{r-1}2^{(r-1)\ell},\quad
E''=\sum_{j=0}^{r-2}\beta_j2^{j\ell}+\beta''_{r-1}2^{(r-1)\ell},
\end{align*}
then the vectors of lower digits are the same (call it $\alpha^*(u)=(\beta_0,\dots,\beta_{r-2})$),
and only the top digits $\beta'_{r-1},\beta''_{r-1}\in\Znl$ may differ.
Thus
\[
\zetaa{k}^{E'}=\zetaa{k}^{\sum_{j=0}^{r-2}\beta_j2^{j\ell}}\zetaa{\ell}^{\beta'_{r-1}},
\qquad
\zetaa{k}^{E''}=\zetaa{k}^{\sum_{j=0}^{r-2}\beta_j2^{j\ell}}\zetaa{\ell}^{\beta''_{r-1}},
\]
and hence
\[
\WHT_f(u)=2^{(m_*-1)/2}v_*
\Bigl(\zetaa{\ell}^{\beta'_{r-1}}-\zetaa{\ell}^{\beta''_{r-1}}\Bigr)\cdot
\zetaa{k}^{\sum_{j=0}^{r-2}\beta_j2^{j\ell}}.
\]
Again the prefactor lies in $\Q(\zetaa{\ell})$, so linear independence implies
that \emph{exactly one} coefficient is nonzero:
\[
C_{\alpha^*(u)}(u)=2^{(m_*-1)/2}v_*
\Bigl(\zetaa{\ell}^{\beta'_{r-1}}-\zetaa{\ell}^{\beta''_{r-1}}\Bigr),
\qquad
C_\alpha(u)=0\ \text{for all }\alpha\neq\alpha^*(u).
\]
\smallskip
We have shown: for every $u\in\V$, there exists $\alpha^*(u)\in(\Znl)^{r-1}$
such that either $\WHT_f(u)=0$ and all $C_\alpha(u)=0$, or else
$C_{\alpha^*(u)}(u)\neq 0$ and $C_\alpha(u)=0$ for $\alpha\neq\alpha^*(u)$.
Now fix $\beta=(\beta_0,\dots,\beta_{r-2})\in(\Znl)^{r-1}$ and consider
$f_\beta=c_{r-1}+\sum_{j=0}^{r-2}\beta_jc_j$. Its Walsh--Hadamard transform satisfies
\begin{equation}
\label{eq:fbeta}
\WHT_{f_\beta}(u)
=\sum_{x\in\V}\zetaa{\ell}^{c_{r-1}(x)+\sum_{j=0}^{r-2}\beta_jc_j(x)}(-1)^{\langle u,x\rangle}
=\sum_{\alpha\in(\Znl)^{r-1}}\zetaa{\ell}^{\sum_{j=0}^{r-2}\beta_j\alpha_j}\,C_\alpha(u).
\end{equation}
Similarly, in general for any $F:(\Znl)^{r-1}\to\Z_{2^m}$ and any $u\in\V$, 
using $\zetaa{\ell}^{2^{\ell-m}s}=\zeta_{2^m}^s$ gives
\[
\WHT_{f_F}(u)
= \sum_{\alpha\in(\Znl)^{r-1}}\zeta_{2^m}^{F(\alpha)}\,C_\alpha(u),
\]
By the sparsity established above, these sums collapses:
if $\WHT_f(u)=0$ then all $C_\alpha(u)=0$ and $\WHT_{f_\beta}(u)=\WHT_{f_F}(u)=0$.
Otherwise,
\[
\WHT_{f_\beta}(u)=\zetaa{\ell}^{\langle\beta,\alpha^*(u)\rangle}\,C_{\alpha^*(u)}(u),
\text{ so }
|\WHT_{f_\beta}(u)|=|C_{\alpha^*(u)}(u)|=|\WHT_f(u)|
\]
and
\[
\WHT_{f_F}(u)=\zeta_{2^m}^{F({\alpha^*(u)})}\,C_{\alpha^*(u)}(u),
\text{ so }
|\WHT_{f_F}(u)|=|C_{\alpha^*(u)}(u)|=|\WHT_f(u)|.
\]

Consequently, for every $u\in\V$ and every $\beta\in(\Znl)^{r-1}$ we have
$|\WHT_{f_F}(u)|=|\WHT_{f_\beta}(u)|\in \{0\}\cup\Lambda$ and in fact the multiset of
magnitudes of $\WHT_{f_\beta}$ and in general of $\WHT_{f_F}(u)$ equal that of $\WHT_f$.
Therefore each $f_\beta, f_F$ is landscape over $\Znl$ with the same levels as $f$,
proving~\textup{(i)}. Moreover $|\WHT_{f_\beta}(u)|=|\WHT_{f_F}(u)|=|\WHT_f(u)|$ for all $u$,
which implies~\textup{(ii)}.
\end{proof}

\begin{remark}
The assumption $\ell\ge 2$ ensures that in the odd-level case the two-term
representation of $\sqrt2$ over $\Q(\zeta_{2^k})$ collapses into a single basis element
over $\Q(\zeta_{2^\ell})$, so no additional landscape levels arise in the
derived functions $f_\beta$. Note that in \cite{MTQ17}, it is proven for $\ell=1$ and
$s$-gplateaued functions that if $f$ is $s$-gplateaued and $n+s$ is odd (corresponding
to $t=1,v_1=1$ and $m_1$ odd), then $f_\beta$ is $(s+1)$-gplateaued. In our case,
however, the magnitude is preserved.
\end{remark}

We now present an  unconditional 
sufficiency condition 
in terms of a small subset of the full family of derived functions indexed by arbitrary maps $F:(\Znl)^{r-1}\to\Z_{2^m}$, $2\leq m\leq\ell$.

\begin{theorem}
\label{thm:landscape_iff}
Let $k=\ell r$ with $\ell \geq 2$ and $r \geq 2$, and let $m$ be  fixed such that $2\leq m\leq\ell$.
For a function $F: (\Znl)^{r-1} \to \Z_{2^m}$, we denote $f_F: \V \to \Znl$ by
$f_F(x) = c_{r-1}(x) + 2^{\ell-m}\,F(c_0(x),\ldots,c_{r-2}(x))$.
If  $c_{r-1}$ is landscape over $\Znl$ and, for every $F=a\cdot\mathbf{1}_{\{\alpha^{(0)}\}}$, $\alpha^{(0)}\in(\Znl)^{r-1}$,  
with $a\in\{1,2^{m-1}\}$,  
the function $f_F$ is landscape over $\Znl$
with $|\WHT_{f_F}(u)| = |\WHT_{c_{r-1}}(u)|$ for all $u\in\V$, then  $f$ is landscape over $\Znk$ with the same levels, and  $|\WHT_{f}(u)| = |\WHT_{c_{r-1}}(u)|$ for all $u\in\V$.
\end{theorem}

\begin{remark}
The most natural value of $m$ to choose is $m=\ell$, where $F:(\Znl)^{r-1}\to\Znl$ and
$f_F=c_{r-1}+F(c_0,\ldots,c_{r-2})$ is a direct generalization of the binary
characterization in Theorem~\textup{\ref{thm:binary_char}}. 
Compared with Theorem~\textup{\ref{thm:binary_char}}, which requires checking
$2^{2^{k-1}}$ Boolean functions $F:\mathbb{F}_2^{k-1}\to\mathbb{F}_2$,
the sufficiency condition here requires only $2\cdot 2^{\ell(r-1)} = 2^{k-\ell+1}$
one-hot functions $F = a\cdot\mathbf{1}_{\{\alpha^{(0)}\}}$, with
$\alpha^{(0)}\in(\Znl)^{r-1}$ and $a\in\{1,2^{m-1}\}$, plus the function $c_{r-1}$.
This represents an exponential reduction in the number of functions to check:
$2^{k-\ell+1}+1$ versus $2^{2^{k-1}}$. 
 In fact, we only need $\alpha^{(0)}\in \mathrm{Image}(c_0,\ldots,c_{r-2})$ (image of the argument), and,  as we will see in Remark~\textup{\ref{rem:singleton_cells}}, the $c_0,\ldots,c_{r-2}$ cannot be injective for $f$ to be landscape, and therefore we only need to check $1+2|\mathrm{Image}(c_0,\ldots,c_{r-2})|<2^{k-\ell+1}+1$.
\end{remark}

\begin{proof}[Proof of Theorem~\textup{\ref{thm:landscape_iff}}]
For any $F:(\Znl)^{r-1}\to\Z_{2^m}$ and any $u\in\V$, grouping
$\WHT_{f_F}(u)$ by the partition $\{P_\alpha\}$ of
Definition~\ref{def:partition} and using $\zetaa{\ell}^{2^{\ell-m}s}=\zeta_{2^m}^s$
gives
\[
\WHT_{f_F}(u)
= \sum_{\alpha\in(\Znl)^{r-1}}\zeta_{2^m}^{F(\alpha)}\,C_\alpha(u),
\]
where $C_\alpha(u)=\sum_{x\in P_\alpha}\zetaa{\ell}^{c_{r-1}(x)}
(-1)^{\langle u,x\rangle}$.

Assume $c_{r-1}$ is landscape over $\Znl$, and every function $F:\mathbb{F}_2^{k-1}\to\mathbb{F}_2$ defined by $F=a\cdot\mathbf{1}_{\{\alpha^{(0)}\}}$, $\alpha^{(0)}\in(\Znl)^{r-1}$, and $a\in\{1,2^{m-1}\}$ 
is landscape over $\Znl$ with
$|\WHT_{f_F}(u)|=|\WHT_{c_{r-1}}(u)|$ for all $u\in\V$.
Fix $u\in\V$, write $W:=\WHT_{c_{r-1}}(u)=\sum_\alpha C_\alpha(u)$, and $M:=|W|^2$. 

\medskip
\noindent\textbf{Case 1: $M=0$.}
For each $\alpha^{(0)}\in(\Znl)^{r-1}$, choose  
$a=1$, giving $\WHT_{f_{F^{(0)}}}(u)=W+(\zeta_{2^m}-1)C_{\alpha^{(0)}}(u)$.
The hypothesis gives $|W+(\zeta_{2^m}-1)C_{\alpha^{(0)}}(u)|^2=M=0$,
so  
$(\zeta_{2^m}-1)C_{\alpha^{(0)}}(u)=0$.
Since $\zeta_{2^m}\neq 1$ for $m\geq 1$, we get $C_{\alpha^{(0)}}(u)=0$.
As $\alpha^{(0)}$ was arbitrary, all $C_\alpha(u)=0$ and $\WHT_f(u)=0$.

\medskip
\noindent\textbf{Case 2: $M>0$.}
For each $\alpha^{(0)}\in(\Znl)^{r-1}$ and each $a\in\Z_{2^m}$,
denote by ${F^{(0,a)}}=a\cdot\mathbf{1}_{\{\alpha^{(0)}\}}$, 
so that
$\WHT_{f_{F^{(0,a)}}}(u)=W+(\zeta_{2^m}^a-1)C_{\alpha^{(0)}}(u)$.
The hypothesis gives
$|W+(\zeta_{2^m}^a-1)C_{\alpha^{(0)}}(u)|^2=M$ for all $a\in\Z_{2^m}$.
Expanding and cancelling $|W|^2=M$, we get
\begin{equation}
\label{eq:probe_identity}
2(\cos\theta_a-1)\,\mathrm{Re}(\overline{W}C_{\alpha^{(0)}})
-2\sin\theta_a\,\mathrm{Im}(\overline{W}C_{\alpha^{(0)}})
+(2-2\cos\theta_a)|C_{\alpha^{(0)}}|^2 = 0,
\end{equation}
where $\theta_a=2\pi a/2^m$.
Taking $a=2^{m-1}$ in \eqref{eq:probe_identity},
so $\theta_{2^{m-1}}=\pi$ and $\sin\theta_{2^{m-1}}=0$, we obtain
\begin{equation}
\label{eq:real_part}
\mathrm{Re}(\overline{W}C_{\alpha^{(0)}}) = |C_{\alpha^{(0)}}|^2.
\end{equation} 
Substituting \eqref{eq:real_part} into \eqref{eq:probe_identity} gives
$-2\sin\theta_a\,\mathrm{Im}(\overline{W}C_{\alpha^{(0)}})=0$.
Taking $a=1$, and since $\sin(2\pi/2^m)>0$
for $m> 1$,  this gives
 \begin{equation}
 \label{eq:imag_part}
\mathrm{Im}(\overline{W}C_{\alpha^{(0)}}) = 0.
\end{equation}
Together, \eqref{eq:real_part} and \eqref{eq:imag_part} give
$\overline{W}C_{\alpha^{(0)}}=|C_{\alpha^{(0)}}|^2\in\R_{\geq 0}$ for every $\alpha^{(0)}$.
Since $W\neq 0$, this means $C_{\alpha^{(0)}}=\lambda_{\alpha^{(0)}}W$ where
$\lambda_{\alpha^{(0)}}:=|C_{\alpha^{(0)}}|^2/|W|^2\geq 0$.
Substituting back into \eqref{eq:real_part}:
$\lambda_{\alpha^{(0)}}|W|^2=\lambda_{\alpha^{(0)}}^2|W|^2$,
so $\lambda_{\alpha^{(0)}}(\lambda_{\alpha^{(0)}}-1)=0$, giving
$\lambda_{\alpha^{(0)}}\in\{0,1\}$ for each $\alpha^{(0)}$.
Finally, summing over all $\alpha^{(0)}$ and using $\sum_{\alpha^{(0)}}C_{\alpha^{(0)}}=W$, we get
\[
W = \sum_{\alpha^{(0)}}\lambda_{\alpha^{(0)}}W = W\sum_{\alpha^{(0)}}\lambda_{\alpha^{(0)}},
\]
so $\sum_{\alpha^{(0)}}\lambda_{\alpha^{(0)}}=1$ (since $W\neq 0$).
Since each $\lambda_{\alpha^{(0)}}\in\{0,1\}$ and they sum to $1$, exactly one
index $\alpha^*(u)$ satisfies $\lambda_{\alpha^*(u)}=1$, giving
$C_{\alpha^*(u)}(u)=W$ and $C_\alpha(u)=0$ for all $\alpha\neq\alpha^*(u)$.

\medskip
In both cases, for every $u\in\V$ at most one $C_\alpha(u)$ is nonzero.
If all $C_\alpha(u)=0$ then $\WHT_f(u)=0$.
Otherwise the unique nonzero coefficient satisfies
$C_{\alpha^*(u)}(u)=W=\WHT_{c_{r-1}}(u)$, so
\[
|\WHT_f(u)|
=\left|C_{\alpha^*(u)}(u)\cdot
\zetaa{k}^{\sum_{j=0}^{r-2}\alpha^*_j(u)2^{j\ell}}\right|
=|C_{\alpha^*(u)}(u)|=|\WHT_{c_{r-1}}(u)|\in\Lambda,
\]
where $\Lambda$ is the landscape level set of $c_{r-1}$.
Since $u\in\V$ was arbitrary and $|\WHT_f(u)|\in\{0\}\cup\Lambda$ for all
$u\in\V$, with $\Lambda$ consisting of values of the form $2^{m_i/2}v_i$
inherited from the landscape property of $c_{r-1}$ over $\Znl$,
it follows that $f$ is landscape over $\Znk$ with the same levels.
\end{proof}

\begin{corollary}
\label{cor:sparsity}
Let $k = \ell r$ with $\ell \geq 2$ and $r \geq 2$, and let 
$f = \sum_{j=0}^{r-1} 2^{j\ell} c_j : \V \to \Znk$ 
be a generalized Boolean function. For each $u \in \V$ and 
$\alpha \in (\Znl)^{r-1}$, define $
C_\alpha(u) = \sum_{x \in P_\alpha} \zetaa{\ell}^{c_{r-1}(x)} (-1)^{\langle u, x \rangle},
\quad\text{where}\quad
P_\alpha = \{x \in \V : c_j(x) = \alpha_j \text{ for } 0 \leq j \leq r-2\}.
$
Then $f$ is landscape with levels in $\Lambda = \{2^{m_1/2}v_1,\ldots,2^{m_t/2}v_t\}$ if and only if, for every $u \in \V$, either $C_\alpha(u)=0$ for all $\alpha \in (\Znl)^{r-1}$, or there is exactly one nonzero $\alpha^*(u)$ and $|C_{\alpha^*(u)}(u)|\in\Lambda$.
\end{corollary}

\begin{proof}
$\Rightarrow$ (necessity): This follows directly from the proof of Theorem~\ref{thm:landscape_necessity}.

$\Leftarrow$ (sufficiency): Let $u\in\V$. We can distinguish between two cases.
First, if $C_\alpha(u)=0$ for all $\alpha \in (\Znl)^{r-1}$, then $\WHT_f(u) = \sum_{\alpha \in (\Znl)^{r-1}} C_\alpha(u) \cdot 
\zetaa{k}^{\sum_{j=0}^{r-2} \alpha_j 2^{j\ell}}=0.$
If there is exactly one nonzero $\alpha^*(u)$ and $|C_{\alpha^*(u)}(u)|\in\Lambda$, then
$\left|\WHT_f(u)\right| = \left|\sum_{\alpha \in (\Znl)^{r-1}} C_\alpha(u) \cdot 
\zetaa{k}^{\sum_{j=0}^{r-2} \alpha_j 2^{j\ell}}\right|
= \left|C_{\alpha^*(u)}(u) \cdot 
\zetaa{k}^{\sum_{j=0}^{r-2} (\alpha^*(u))_j 2^{j\ell}}\right|
=|C_{\alpha^*(u)}(u)|\in\Lambda$.
The claim is shown.
\end{proof}

\begin{remark}
\label{rem:singleton_cells}
A consequence of Corollary~\textup{\ref{cor:sparsity}} is that if there exist distinct $\alpha,\alpha' \in (\Znl)^{r-1}$ with $|P_\alpha| = |P_{\alpha'}| = 1$, then $f$  cannot be landscape: for any $u$ with $\langle u, x_\alpha - x_{\alpha'} \rangle = 0$ 
(where $P_\alpha = \{x_\alpha\}$, $P_{\alpha'} = \{x_{\alpha'}\}$), both cells  contribute nonzero terms $C_\alpha(u)$ and $C_{\alpha'}(u)$, violating sparsity.  In particular, if all lower components $c_0,\ldots,c_{r-2}$ are injective, then  every partition cell has size at most one and landscape is impossible. This  illustrates why the structural assumptions in Theorems~\textup{\ref{thm:plateaued_common_arg}} 
are genuinely necessary, as confirmed by Example~\textup{\ref{ex:counterexample_3to1}}.
\end{remark}

As a consequence, we can characterize gbent functions:

\begin{corollary}
\label{cor:iff}
Let $k = \ell r$ with $\ell \geq 2$ and $r \geq 2$, and let
$f: \V \to \Znk$ have $2^\ell$-adic decomposition
$f = \sum_{j=0}^{r-1} 2^{j\ell} c_j$.
For each $u \in \V$ and $\alpha \in (\Znl)^{r-1}$, define
$C_\alpha(u) = \sum_{x \in P_\alpha}
\zetaa{\ell}^{c_{r-1}(x)} (-1)^{\langle u,x \rangle}$.
Then $f$ is gbent over $\Znk$ if and only if for each $u \in \V$,
exactly one coefficient $C_\alpha(u)$ is nonzero.
\end{corollary}

\begin{proof}
Follows directly from Corollary \ref{cor:sparsity} taking $t=1,m_1=n,v_1=1$, as by Parseval the case $C_\alpha(u)=0$ for all $\alpha \in (\Znl)^{r-1}$ is excluded, since then $|\WHT_f(u)|=0$.
\end{proof}

While Theorem~\ref{thm:landscape_iff} already provides a 
sufficiency condition via a small subset of the full family of maps $F:(\Znl)^{r-1}\to\Z_{2^m}$, it is also useful to have sufficiency results for the smaller affine
subfamily $\{f_\beta\}_{\beta\in(\Znl)^{r-1}}$ alone, which involves only $2^{k-\ell}$ functions rather than $2^{k-\ell+1}+1$. As Example~\ref{ex:counterexample_3to1} shows, the affine subfamily alone does not suffice in general, and additional structural assumptions are genuinely necessary. The following  theorem identify precise conditions under which the affine subfamily characterizes  gplateaued functions. 
 
\begin{theorem}
\label{thm:plateaued_common_arg}
Let $n$ be an even integer and $k=\ell r$ with $\ell \ge 2$.
Let $f:\V\to\Znk$ have $2^\ell$-adic components $c_0,\dots,c_{r-1}$.
Suppose that for every $\beta\in(\Znl)^{r-1}$, the function
$g_\beta(x)=c_{r-1}(x)+\sum_{j=0}^{r-2}\beta_j\,c_j(x)$
is $s$-gplateaued over $\Znl$. Additionally, suppose that for every $u\in\V$,
letting
\[
\mathcal S(u)=\{\alpha\in(\Znl)^{r-1}: C_\alpha(u)\neq 0\},
\qquad
D(u)=\mathcal S(u)-\mathcal S(u),
\]
the following hold:
\begin{enumerate}
\item[\textup{(a)}] the nonzero coefficients $(C_\alpha(u))_{\alpha\in\mathcal S(u)}$
satisfy the common-argument hypothesis (Definition~\textup{\ref{def:common_argument}}); and
\item[\textup{(b)}] $\Stab\big(D(u)+D(u)\big)=\{0\}$ in $(\Znl)^{r-1}$.
\end{enumerate}
Then $f$ is $s$-gplateaued over $\Znk$.
\end{theorem}

\begin{proof}
Fix $u\in\V$ and set
\[
\mathcal S(u)=\{\alpha\in(\Znl)^{r-1}: C_\alpha(u)\neq 0\}.
\]
By hypothesis, each $g_\beta$ is $s$-gplateaued over $\Znl$, i.e.
\[
|\WHT_{g_\beta}(u)|\in\{0,2^{(n+s)/2}\}\qquad\text{for all }\beta\in(\Znl)^{r-1}.
\]
By \eqref{eq:fbeta} we have
\[
\WHT_{g_\beta}(u)
=\sum_{\alpha\in(\Znl)^{r-1}}\zetaa{\ell}^{\sum_{j=0}^{r-2}\beta_j\alpha_j}\,C_\alpha(u),
\]
so $\WHT_{g_\beta}(u)$ is the Fourier transform of the finitely supported measure
\[
\mu_u=\sum_{\alpha\in\mathcal S(u)} C_\alpha(u)\,\delta_\alpha
\]
on the finite abelian group $G=(\Znl)^{r-1}$, evaluated at the character
$\chi_\beta(\alpha)=\zetaa{\ell}^{\langle\beta,\alpha\rangle}$.
For any $\alpha\neq0,$
as $\beta$ ranges over $(\Znl)^{r-1}$, the characters $\chi_\beta$ exhaust
$\widehat G$.

By (a), the nonzero coefficients of $\mu_u$ satisfy the common-argument hypothesis, and by
$s$-gplateauedness of all $g_\beta$ we have
$|\widehat\mu_u(\chi)|\in\{0,A\}$ for all $\chi\in\widehat G$, where $A=2^{(n+s)/2}$.
Therefore Theorem~\ref{thm:overconstrained} applies to $\mu_u$ on $G$ with this value of $A$,
yielding
\[
|\overline{\mathcal S}(u)|\le 2 \quad\text{in } G/H,
\qquad
H=\Stab\big(D(u)+D(u)\big),
\qquad
D(u)=\mathcal S(u)-\mathcal S(u).
\]
By (b) we have $H=\{0\}$, hence $G/H=G$ and consequently
$|\mathcal S(u)|\le 2$.

We claim in fact that $|\mathcal S(u)|\le 1$.
Indeed, suppose for contradiction that $|\mathcal S(u)|=2$, say
$\mathcal S(u)=\{\alpha,\alpha'\}$ with $\alpha\neq\alpha'$.
Write $z_1=C_\alpha(u)$, $z_2=C_{\alpha'}(u)$ and $\gamma=\alpha'-\alpha\in G$.
For $\chi\in\widehat G$ set
\[
S_\chi=z_1\chi(\alpha)+z_2\chi(\alpha').
\]
Then $S_{\chi_\beta}=\WHT_{g_\beta}(u)$ for all $\beta$, so by hypothesis
$|S_\chi|\in\{0,A\}$ for all $\chi\in\widehat G$.
Applying Theorem~\ref{thm:two_point_analysis} to $(\alpha,\alpha')$ in $G$,
the cases $\ord(\gamma)\ge 4$ are impossible (they would yield at least two distinct
nonzero magnitudes), and the exceptional $\ord(\gamma)=3$ case is also impossible under
the common-argument hypothesis (a), since then $z_1/z_2\in\R$ and hence
$-z_1/z_2\in\R$, which cannot be a primitive cube root of unity.
Therefore $\ord(\gamma)=2$.
But if $\ord(\gamma)=2$, then $D(u)=\{0,\pm\gamma\}=\{0,\gamma\}$ and hence
$D(u)+D(u)=\{0,\gamma\}$,
which is stabilized by translation by $\gamma$; 
$\{0,\gamma\}+\gamma=\{\gamma,0\}=\{0,\gamma\}$.
Thus $\gamma\in\Stab(D(u)+D(u))$ is a nonzero stabilizer, contradicting (b).
Hence $|\mathcal S(u)|\neq 2$, so indeed $|\mathcal S(u)|\le 1$.
If $|\mathcal S(u)|=0$, then all $C_\alpha(u)=0$ and \eqref{eq:WHT_decomp} gives
$\WHT_f(u)=0$.
If $|\mathcal S(u)|=1$, say $\mathcal S(u)=\{\alpha^*\}$, then
\[
\WHT_{g_0}(u)=\sum_{\alpha\in(\Znl)^{r-1}} C_\alpha(u)=C_{\alpha^*}(u),
\]
so $|C_{\alpha^*}(u)|\in\{0,A\}$.
Therefore, by Corollary \ref{cor:sparsity} ($t=1$), $f$ is $s$-gplateaued over $\Znk$.
\end{proof}
 
Specializing the sufficiency conditions to the gbent case ($s=0$) allows us to drop the stabilizer requirement, yielding the following characterization.

\begin{corollary}
\label{cor:main}
Let $n$ be an even integer and $k=\ell r$ with $\ell \geq 2$. Let $f: \V \to \Znk$ be a function with $2^\ell$-adic components
$c_0, \dots, c_{r-1}$. Then, if $c_{r-1}$ is gbent over $\Znl$, and for each $u \in \V$ the nonzero coefficients $C_\alpha(u)$ satisfy the common-argument hypothesis, then $f$ is gbent over $\Znk$.
\end{corollary}
\begin{proof}
Note that $c_{r-1}$ being gbent is necessary by
Theorem~\ref{thm:landscape_necessity} with $t=1$. For sufficiency, fix $u \in \V$. 
Since $c_{r-1}$ is gbent, $|\sum_{\alpha} C_\alpha(u)| = 2^{n/2}$. Under the common-argument hypothesis, $|\sum_{\alpha} C_\alpha(u)| = \sum_{\alpha} |C_\alpha(u)|$, so $\sum_{\alpha} |C_\alpha(u)| = 2^{n/2}$.  The Walsh transform of $f$ is $\WHT_f(u) = \sum_{\alpha} C_\alpha(u) \zeta_{2^k}^{E_\alpha}$, where $E_\alpha = \sum_{j=0}^{r-2} \alpha_j 2^{j\ell}$. By the triangle inequality, $|\WHT_f(u)| \leq \sum_{\alpha} |C_\alpha(u)| = 2^{n/2}$ for all $u \in \V$. By Parseval's identity, $\sum_{u \in \V} |\WHT_f(u)|^2 = 2^{2n}$.  Since there are $2^n$ terms in the sum and each satisfies $|\WHT_f(u)|^2 \leq 2^n$,  equality must hold for all $u$. Therefore, $|\WHT_f(u)| = 2^{n/2}$ for all $u$, 
so $f$ is gbent over $\Znk$.
\end{proof}
 We next show an example that illustrates why additional conditions are necessary, even in the case of gbents, and that even the gbentness of the $g_\beta$ in itself is not enough.
\begin{example}
\label{ex:counterexample_3to1}
Take $n=2$, $\ell=2$, $r=2$, so $k=4$ and $\V=\F^2$.
Let $c_0,c_1:\V\to\Z_4$ be defined by
\[
\begin{array}{c|cccc}
x & (0,0) & (0,1) & (1,0) & (1,1)\\\hline
c_0(x) & 0 & 0 & 1 & 3\\
c_1(x) & 0 & 1 & 0 & 3
\end{array}
\]
and set $f(x)=c_0(x)+4c_1(x)\in\Z_{16}$.
For $\beta\in\Z_4$, define $g_\beta(x)=c_1(x)+\beta\,c_0(x)\in\Z_4$.

\noindent\textbf{Claim~1.} For every $\beta\in\Z_4$, the function $g_\beta$
is gbent over $\Z_4$, with $|\WHT_{g_\beta}(u)|=2$ for all $u\in\V$.
This is verified by direct computation.

\noindent\textbf{Claim~2.} The function $f:\V\to\Z_{16}$ is not gbent.
Direct computation gives
\[
\bigl(|\WHT_f(0,0)|,\,|\WHT_f(0,1)|,\,|\WHT_f(1,0)|,\,|\WHT_f(1,1)|\bigr)
=\bigl(3.018\ldots,\,1.027\ldots,\,1.311\ldots,\,2.029\ldots\bigr),
\]
which is not constant, so $f$ is not gbent.

Therefore all $g_\beta$ are gbent (and therefore have  identical magnitudes), but $f$ is not gbent. The common-argument hypothesis fails here, and the partition cells $P_0=\{(0,0),(0,1)\}$ and $P_1=\{(1,0)\}$ and $P_3=\{(1,1)\}$ each have sizes $2,1,1$ respectively, so by Remark~\textup{\ref{rem:singleton_cells}} the existence of two singleton cells already prevents $f$ from being landscape. This confirms that additional conditions as 
the common-argument hypothesis in Theorem~\textup{\ref{thm:plateaued_common_arg}} 
are genuinely necessary.
\end{example}

\subsection{Verification strategies and complexity}
\label{sec:verification}

The $2^\ell$-adic characterization enables multiple verification strategies with
varying computational costs and structural requirements.
Theorem~\ref{thm:landscape_iff} provides a complete characterization of landscape
functions via a family of maps $F:(\Znl)^{r-1}\to\Z_{2^m}$, for any $1\leq m\leq\ell$.
The size of this family is less than $2^{\ell(r-1)+1}+1$, regardless of $m$, compared to the $2^{k-1}$ Boolean functions
required by Theorem~\ref{thm:binary_char}.
Choosing instead the \emph{affine subfamily}
$\{f_\beta\}_{\beta\in(\Znl)^{r-1}}$ (i.e., $F$ linear, corresponding to $m=\ell$
and $F(\alpha)=\langle\beta,\alpha\rangle$) reduces the count to
$2^{\ell(r-1)}=2^{k-\ell}$ checks, under the common-argument assumption.

The necessity direction --- that gbentness of $f$ implies gbentness of all $f_\beta$ and $f_F$
--- is unconditional and follows from Theorem~\ref{thm:landscape_necessity} and
Corollary~\ref{cor:main}(i). For the converse, Example~\ref{ex:counterexample_3to1}
demonstrates that gbentness of all $f_\beta$ does not imply gbentness of $f$ in general, so additional structural hypotheses are genuinely needed.

\begin{remark}
\label{rem:affine_not_sufficient}
The common-argument hypothesis in
\textup{\ref{thm:plateaued_common_arg}} cannot be omitted, even when
lower components are affine. When $c_0, \ldots, c_{r-2}$ are
$\mathbb{Z}_{2^\ell}$-affine, the partition cells $P_\alpha = x_\alpha + H$ are
affine cosets where $H = \ker(c_0, \ldots, c_{r-2})$. This gives
\[
C_\alpha(u) = (-1)^{\langle u, x_\alpha \rangle}
\sum_{h \in H} \zetaa{\ell}^{c_{r-1}(x_\alpha + h)} (-1)^{\langle u, h \rangle}.
\]
While the factor $(-1)^{\langle u, x_\alpha \rangle} \in \{\pm 1\}$ is real,
the sum over $H$ can have arbitrary complex argument depending on $c_{r-1}$.
Thus affine lower components alone do not guarantee common-argument, and
Example~\textup{\ref{ex:counterexample_3to1}} shows that without additional constraints gbentness
of all $f_\beta$ need not imply gbentness of $f$.
\end{remark}

\begin{proposition}
\label{thm:iff_basis}
If, for each $u \in \V$ the nonzero coefficients $C_\alpha(u)$ satisfy the common-argument hypothesis, and if  the lower components $c_0,\dots,c_{r-2}$ are $\Znl$-affine, 
$f$ is gbent over $\Znk$ if and only if all $r-1$ basis functions $f_j=c_{r-1}+c_j$, $j=0,\ldots,r-2$, are gbent over $\Znl$.
\end{proposition}

\begin{proof}
The necessity follows from Theorem~\ref{thm:landscape_necessity} by taking
$\beta=e_j$, for $j=0,\ldots,r-2$, since then
$f_{e_j}=c_{r-1}+c_j=f_j$.

For the converse, assume that all $r-1$ basis functions
$f_j=c_{r-1}+c_j$, $j=0,\ldots,r-2$,
are gbent over $\Znl$. Since each lower component $c_j$ is
$\Znl$-affine, it is gbent over $\Znl$. Fix any $j\in\{0,\ldots,r-2\}$. Then
$c_{r-1}=f_j-c_j=f_j+(-c_j)$.
Because $f_j$ is gbent and $-c_j$ is affine, it follows that
$c_{r-1}$ is gbent over $\Znl$.
By hypothesis, for each $u\in\V$, the nonzero coefficients
$C_\alpha(u)$ satisfy the common-argument hypothesis. Therefore all assumptions of Corollary~\ref{cor:main} are satisfied, and we conclude that $f$ is gbent over $\Znk$.
\end{proof}
 
Table~\textup{\ref{tab:verification_counts}} summarizes the computational cost of each verification strategy for gbent functions.

{\small
\begin{table}[h]
\centering
{\small
\caption{Verification complexity for gbent functions $f:\V\to\Znk$,
$k = \ell r$}
\label{tab:verification_counts}
\begin{tabular}{|l|l|c|}
\hline
\textbf{Method} & \textbf{Assumptions} & \textbf{Functions checked} \\ \hline
Binary decomposition~\cite{mtqwwf}
  & None (necessity+sufficiency)
  & $2^{2^{k-1}}$ \\ 
\hline
$2^\ell$-adic family, (Thm.~\ref{thm:landscape_iff})
  & None (sufficiency)
  & $< 2^{k-\ell+1}+1$ \\ \hline
Higher component (Cor.~\ref{cor:main})
  & \begin{tabular}{@{}l@{}}
      Common-argument
    \end{tabular}
  & 1 \\ \hline
\end{tabular}
}
\end{table}
}
\medskip
 
For instance, when $(k,\ell)=(12,4)$, verification drops from $2^{2^{11}}$ Boolean transforms to $2^9$ $2^\ell$-adic family checks, or to just $1$ function under common-argument assumptions.   
 
The underlying spectral interpretation is provided by
Corollary~\ref{cor:iff}: gbentness of $f$ is equivalent to $1$-sparsity of the coefficients $C_\alpha(u)$  (that is, the existence of a unique $C_\alpha(u)\neq0)$ for each $u\in\V$.
The verification strategies enforce this sparsity either globally ($2^\ell$-adic family) or economically (basis test under phase alignment). Overall, the $2^\ell$-adic framework yields
exponential reductions in verification complexity when structural conditions are met, with the strongest savings arising when the
common-argument hypothesis are present.

\section{Structural properties of the decomposition}
\label{sec:properties}

Having established the $2^\ell$-adic characterization of gbent, gplateaued, and landscape functions, we now investigate how this decomposition interacts with key structural and cryptographic properties. We show that the decomposition preserves duality and derivative properties, while enabling precise analysis of differential spectra, and it shows a predictable behaviour for the Maiorana-McFarland class structure.

\subsection{Interaction with duality}

The dual function of a gbent function has a natural interpretation in terms of the $2^\ell$-adic decomposition.

\begin{proposition}
\label{prop:dual}
Let $f: \V \to \Znk$ be a gbent function with dual $\widetilde{f}$. Let $\{c_j\}$ and $\{\widetilde{c}_j\}$ be their respective $2^\ell$-adic components. For each $u \in \V$, the lower $r-1$ components of $\widetilde{f}(u)$ are determined by the unique nonzero partition cell $P_{\alpha^*(u)}$ of $f$,
$(\widetilde{c}_0(u), \ldots, \widetilde{c}_{r-2}(u)) = \alpha^*(u)$.
\end{proposition}

\begin{proof}
Since $f$ is gbent, Corollary~\ref{cor:iff} gives that for each $u\in\V$
there exists a unique $\alpha^*(u)\in(\Znl)^{r-1}$ such that
$C_{\alpha^*(u)}(u)\neq 0$ and $C_\alpha(u)=0$ for all $\alpha\neq\alpha^*(u)$
(when $\WHT_f(u)\neq 0$; the case $\WHT_f(u)=0$ cannot occur since $f$ is gbent).
From Equation~\eqref{eq:WHT_decomp},
\[
\WHT_f(u)
= C_{\alpha^*(u)}(u)\cdot\zetaa{k}^{\sum_{j=0}^{r-2}\alpha_j^*(u)2^{j\ell}}.
\]
Since $f$ is gbent and $k\geq 4$, Theorem~\ref{reg-landscape} gives
$\WHT_f(u)=\epsilon_u 2^{n/2}\zetaa{k}^{\widetilde{f}(u)}$
for some $\epsilon_u\in\{\pm 1,\pm i\}$.
Since $C_{\alpha^*(u)}(u)\in\Q(\zetaa{\ell})$ and
$\zetaa{k}^{\sum_{j=0}^{r-2}\alpha_j^*(u)2^{j\ell}}$ is a basis element
over $\Q(\zetaa{\ell})$ by Lemma~\ref{lem:basis}, comparing with
$\epsilon_u 2^{n/2}\zetaa{k}^{\widetilde{f}(u)}$ and writing
$\widetilde{f}(u)=\sum_{j=0}^{r-1}\widetilde{c}_j(u)2^{j\ell}$,
linear independence forces
\[
\sum_{j=0}^{r-2}\alpha_j^*(u)2^{j\ell}
\equiv\sum_{j=0}^{r-2}\widetilde{c}_j(u)2^{j\ell}\pmod{2^{(r-1)\ell}},
\]
hence $\widetilde{c}_j(u)=\alpha_j^*(u)$ for $j=0,\ldots,r-2$.
\end{proof}

 \subsection{Maiorana--McFarland functions}

The Maiorana--McFarland construction is one of the most important classes of
generalized bent functions. Motivated by the generalized modular-valued setting
of Kumar, Scholtz, and Welch~\cite{ksw}, we consider here the following
\emph{binary-input specialization} 
$f(x,y)=\langle x,\pi(y)\rangle+g(y)$,  $(x,y)\in \F^m\times \F^m$,
where $\pi:\F^m\to\F^m$ is a permutation, $g:\F^m\to\Znk$, and
$\langle x,\pi(y)\rangle=\sum_{i=1}^m x_i\pi(y)_i \pmod{2^k}$,
with $x_i,\pi(y)_i\in\{0,1\}$. Thus the integer
$N(x,y):=\sum_{i=1}^m x_i\pi(y)_i\in\{0,\ldots,m\}$
is well-defined before reduction modulo $2^k$, and
$\langle x,\pi(y)\rangle\equiv N(x,y)\pmod{2^k}$.
We emphasize that, in this binary-input specialization, higher $2^\ell$-adic
digits may be nontrivial whenever $N(x,y)\ge 2^\ell$, and carries may occur when
adding $N(x,y)$ and $g(y)$.

We now continue with the consequences of the general $2^\ell$-adic theory for this
class.

\begin{theorem}
\label{thm:MM}
Let
$f(x,y)=\langle x,\pi(y)\rangle+g(y)$
be a gbent function from $\F^{2m}\to\Znk$, where $\pi:\F^m\to\F^m$ is a
permutation, $g:\F^m\to\Znk$, and
$\langle x,\pi(y)\rangle=\sum_{i=1}^m x_i\pi(y)_i\pmod{2^k}$ with
$x_i,\pi(y)_i\in\{0,1\}$.
Let
$f=\sum_{j=0}^{r-1}2^{j\ell}c_j$
 and 
$g=\sum_{j=0}^{r-1}2^{j\ell}g_j$
be the $2^\ell$-adic decompositions of $f$ and $g$, respectively, where
$c_j:\F^{2m}\to\Znl$ and $g_j:\F^m\to\Znl$.
Then:
\begin{enumerate}
\item[\textup{(i)}] The components $c_j$ are the $2^\ell$-adic digits of
$f(x,y)\equiv N(x,y)+g(y)\pmod{2^k}$, hence are well-defined functions
$c_j:\F^{2m}\to\Znl$. If $\phi_j(x,y)\in\Znl$ denotes the $j$-th $2^\ell$-adic
digit of the integer $N(x,y)$ alone, then in general
\[
c_j(x,y)\equiv \phi_j(x,y)+g_j(y)+\lfloor{\frac{\phi_{j-1}(x,y)+g_{j-1}(y)}{2^{\ell}}}\rfloor \pmod{2^{\ell}},
\]

\item[\textup{(ii)}] For every $\beta=(\beta_0,\ldots,\beta_{r-2})\in(\Znl)^{r-1}$,
the derived function
$f_\beta(x,y):=c_{r-1}(x,y)+\sum_{j=0}^{r-2}\beta_j\,c_j(x,y)$
is landscape over $\Znl$ with the same set of landscape levels as $f$
(viewed over $\Znk$). In particular, if $f$ is gbent, then every $f_\beta$ is
gbent over $\Znl$.
\end{enumerate}
\end{theorem}

\begin{proof}
For part~\textup{(i)}, each value $f(x,y)\in\Znk$ has a unique $2^\ell$-adic
expansion modulo $2^k$, so the component functions $c_j:\F^{2m}\to\Znl$ are
well-defined. Equivalently, identifying $g(y)\in\Znk$ with its representative in
$\{0,\ldots,2^k-1\}$ and using $N(x,y)\in\{0,\ldots,m\}$, the components are the
$2^\ell$-adic digits of $N(x,y)+g(y)$ modulo $2^k$.

The failure of a componentwise formula can occur because of carries. For example,
take $\ell=1$, $k=2$ (so $r=2$), and suppose $N(x,y)=1$ and $g(y)=1$. Then
$\phi_0=1$, $\phi_1=0$, while $g_0=1$, $g_1=0$. But
$N(x,y)+g(y)=2=0+1\cdot 2$,
so $(c_0,c_1)=(0,1)$, whereas $(\phi_0+g_0,\phi_1+g_1)=(2,0)$ before reduction
in the lower block; the carry from block $0$ changes the higher digit. 
This translates to the formula
$c_j(x,y)\equiv \phi_j(x,y)+g_j(y)+\lfloor{\frac{\phi_{j-1}(x,y)+g_{j-1}(y)}{2^{\ell}}}\rfloor \pmod{2^{\ell}}$.

Part~\textup{(ii)} is an immediate application of
Theorem~\ref{thm:landscape_necessity}\textup{(i)} to the $2^\ell$-adic
decomposition of $f$. Since $f$ is gbent, it is landscape, and hence every
$f_\beta$ is landscape over $\Znl$ with the same levels. In particular, all
$f_\beta$ are gbent over $\Znl$.
\end{proof}

\begin{remark}
The key distinction from the classical Boolean M-M situation is that the
binary-input modular inner product here is not restricted to $\{0,1\}$:
$N(x,y)=\sum_{i=1}^m x_i\pi(y)_i\in\{0,\ldots,m\}$.
Hence its $2^\ell$-adic digits may contribute at several levels, and the addition
$N(x,y)+g(y)$ generally produces inter-block carries. Consequently, one should not
expect a simple formula $c_j=\phi_j+g_j$ without additional hypotheses. Nevertheless, part~\textup{(ii)} shows that the derived functions $f_\beta$ remain gbent unconditionally whenever $f$ is gbent.
\end{remark}

\subsection{Derivatives and algebraic degree}

The $t$-order derivative structure is perfectly preserved by the $2^\ell$-adic decomposition, enabling direct analysis of algebraic degree.

\begin{proposition}
\label{prop:derivatives}
Let $f: \V \to \Znk$ have $2^\ell$-adic components $c_0, \ldots, c_{r-1}$ where $k = \ell r$. Then for all $a_1, \ldots, a_t \in \V$ with $t \geq 1$,
\[
D_{a_1,\ldots,a_t}f(x) = \sum_{j=0}^{r-1} (D_{a_1,\ldots,a_t}c_j(x)) \cdot 2^{j\ell} \pmod{2^k}.
\]
\end{proposition}
\begin{proof}
By definition, $D_{a_1,\ldots,a_t}f(x)$ is computed in the additive group $\Znk$, and the identity $f(z)=\sum_{j=0}^{r-1} c_j 2^{j\ell}$ holds in $\Znk$ for every $z$.
Thus,
\begin{align*}
D_{a_1,\ldots,a_t}f(x) &= \sum_{S \subseteq [t]} (-1)^{|S|} f\left(x + \sum_{i \in S} a_i\right) 
= \sum_{S \subseteq [t]} (-1)^{|S|} \sum_{j=0}^{r-1} c_j\left(x + \sum_{i \in S} a_i\right) 2^{j\ell} \\
&= \sum_{j=0}^{r-1} \left(\sum_{S \subseteq [t]} (-1)^{|S|} c_j\left(x + \sum_{i \in S} a_i\right)\right) 2^{j\ell} 
= \sum_{j=0}^{r-1} D_{a_1,\ldots,a_t}c_j(x) \cdot 2^{j\ell} \pmod{2^k},
\end{align*}
where $[t] = \{1, \ldots, t\}$ and the sum is over all subsets $S \subseteq [t]$.
\end{proof}

\begin{remark}
There is no carry issue in Proposition~\textup{\ref{prop:derivatives}}, since unlike the pointwise $2^\ell$-adic digit decomposition of a sum such as $N(x,y)+g(y)$, the operator $D_{a_1,\ldots,a_t}$  is an additive $\mathbb{Z}$-linear combination of values of $f$ in the abelian group $\Znk$. Hence it commutes with the fixed expansion $f=\sum_{j=0}^{r-1} c_j 2^{j\ell}$.
\end{remark}

\begin{corollary}
\label{cor:second_order}
Let $f: \V \to \Znk$ have $2^\ell$-adic components $c_0, \ldots, c_{r-1}$. Then $f$ satisfies the second-order derivative property (i.e., $D_{a,b}f$ is constant for all $a,b \in \V$) if and only if each component $c_j$ satisfies this property over $\Znl$. Equivalently, $f$ has algebraic degree at most $2$ over $\Znk$ if and only if each $c_j$ has algebraic degree at most $2$ over $\Znl$.
\end{corollary}

\begin{proof}
By Proposition~\ref{prop:derivatives} with $t=2$, we have $D_{a,b}f(x) = \sum_{j=0}^{r-1} (D_{a,b}c_j(x)) \cdot 2^{j\ell}$. The function $f$ satisfies the second-order derivative property if and only if $D_{a,b}f$ is constant for all $a, b$, which holds if and only if each $D_{a,b}c_j$ is constant, equivalent to each $c_j$ having algebraic degree at most $2$.
\end{proof}

\subsection{Differential properties}

We now analyze how the differential spectrum interacts with the $2^\ell$-adic decomposition. We start with a bound for the component differentials.

\begin{proposition}
\label{prop:differential_bound}
Let $f: \V \to \Znk$ have $2^\ell$-adic components $c_0, \ldots, c_{r-1}$. Then for $a \in \V \setminus \{0\}$ and $b \in \Znk$ with $2^\ell$-adic decomposition $b = \sum_{j=0}^{r-1} b_j 2^{j\ell}$,
\[
\Delta_f(a,b) \leq \min_{0 \leq j \leq r-1} \Delta_{c_j}(a, b_j).
\]
\end{proposition}

\begin{proof}
By definition, $\Delta_f(a,b) = |\{x \in \V : f(x+a) - f(x) \equiv b \pmod{2^k}\}|$. 
Write $f(x) = \sum_{j=0}^{r-1} c_j(x) 2^{j\ell}$ and $b = \sum_{j=0}^{r-1} b_j 2^{j\ell}$ with $b_j \in \Znl$. Then
$$
f(x+a) - f(x) = \sum_{j=0}^{r-1} [c_j(x+a) - c_j(x)] 2^{j\ell}.
$$
For this to equal $b$ modulo $2^k$, we need $c_j(x+a) - c_j(x) \equiv b_j \pmod{2^\ell}$ for each $j = 0, \ldots, r-1$, since the $2^\ell$-adic representation is unique modulo $2^k$ when $k = \ell r$.
Therefore,
$$
\{x : f(x+a) - f(x) \equiv b \pmod{2^k}\} = \bigcap_{j=0}^{r-1} \{x : c_j(x+a) - c_j(x) \equiv b_j \pmod{2^\ell}\},
$$
which gives $\Delta_f(a,b) \leq \min_{j} \Delta_{c_j}(a,b_j)$.
\end{proof}

\begin{theorem}
\label{thm:differential_affine}
Let $f: \V \to \Znk$ have $2^\ell$-adic components $c_0, \ldots, c_{r-1}$. For $\beta \in (\Znl)^{r-1}$, let $g_\beta(x) = c_{r-1}(x) + \sum_{j=0}^{r-2}\beta_j c_j(x)$. Then for $a \in \V \setminus \{0\}$ and $b \in \Znl$,
\[
\Delta_{g_\beta}(a,b) = \sum_{\substack{\delta \in (\Znl)^r \\ \delta_{r-1} + \sum_{j=0}^{r-2}\beta_j \delta_j \equiv b \pmod{2^\ell}}} \left|\left\{x \in \V : \forall j, c_j(x+a) - c_j(x) = \delta_j\right\}\right|.
\]
\end{theorem}

\begin{proof}
By definition, $\Delta_{g_\beta}(a,b) = |\{x : g_\beta(x+a) - g_\beta(x) = b \pmod{2^\ell}\}|$. Expanding gives
\[
g_\beta(x+a) - g_\beta(x) = [c_{r-1}(x+a) - c_{r-1}(x)] + \sum_{j=0}^{r-2}\beta_j [c_j(x+a) - c_j(x)].
\]
Let $\delta_j(x) = c_j(x+a) - c_j(x) \in \Znl$ for each $j$. The condition $g_\beta(x+a) - g_\beta(x) = b$ becomes
\[
\delta_{r-1}(x) + \sum_{j=0}^{r-2}\beta_j \delta_j(x) \equiv b \pmod{2^\ell}.
\]
Partitioning $\V$ according to the tuple $\delta(x) = (\delta_0(x), \ldots, \delta_{r-1}(x)) \in (\Znl)^r$, for each fixed $\delta = (\delta_0, \ldots, \delta_{r-1})$, define
\[
S_\delta = \left\{x \in \V : \forall j \in \{0,\ldots,r-1\}, c_j(x+a) - c_j(x) = \delta_j\right\}.
\]
The sets $\{S_\delta\}_{\delta \in (\Znl)^r}$ partition $\V$, so
\[
\Delta_{g_\beta}(a,b) = \sum_{\substack{\delta \in (\Znl)^r \\ \delta_{r-1} + \sum_{j=0}^{r-2}\beta_j \delta_j \equiv b \pmod{2^\ell}}} |S_\delta|,
\]
completing the proof.
\end{proof}

\begin{remark}
Theorem~\textup{\ref{thm:differential_affine}} reveals that the differential properties of $g_\beta$ are completely determined by the joint differential distribution of the component functions. The sets $S_\delta$ capture the simultaneous differential behavior of all components, and the differential of $g_\beta$ is obtained by summing the sizes of those cells where the weighted combination of component differentials equals the target difference.
\end{remark}

\begin{corollary} 
\label{cor:differential_uniform}
If each component $c_j: \V \to \Znl$ is differentially $\delta$-uniform, then each $g_\beta$ is differentially at most $\delta \cdot 2^{\ell(r-1)}$-uniform, and $f$ is at most 
differentially $\delta$-uniform.
\end{corollary}

\begin{proof}
For each $\delta\in(\Znl)^r$, the set $S_\delta$ has size at most $\delta$ by the uniformity assumption on each $c_j$.
For fixed $a$ and $b$, the constraint $\delta_{r-1}+\sum_{j=0}^{r-2}\beta_j\delta_j\equiv b\pmod{2^\ell}$
fixes $\delta_{r-1}$ in terms of $(\delta_0,\ldots,\delta_{r-2})$,
so exactly $2^{\ell(r-1)}$ tuples $\delta$ satisfy it.
Hence $\Delta_{g_\beta}(a,b)\leq \delta\cdot 2^{\ell(r-1)}$.
For $f$, Proposition~\ref{prop:differential_bound} gives
$\Delta_f(a,b)\leq\delta$ directly.
\end{proof}

\section{Relationship to classical S-boxes}
\label{sec:classical-sboxes}

The constructions presented in this paper provide systematic methods for building
generalized Boolean functions with prescribed Walsh spectral properties. A natural
question arises: do existing cryptographic S-boxes, when viewed as generalized Boolean
functions $f: \mathbb{F}_2^n \to \mathbb{Z}_{2^k}$ via the canonical binary
decomposition $f(x) = \sum_{i=0}^{k-1} f_i(x) \cdot 2^i$, already possess these
properties? In this section, we provide empirical evidence that the answer is negative,
thereby validating the necessity of our constructive approach.

We analyzed several widely-used cipher S-boxes, including PRESENT~\cite{present},
GIFT~\cite{gift}, PRINCE~\cite{prince}, and SKINNY~\cite{skinny}. For each S-box
$S: \mathbb{F}_2^4 \to \mathbb{F}_2^4$, we computed its Walsh-Hadamard transform as
a function $f: \mathbb{F}_2^4 \to \mathbb{Z}_{16}$.
We then examined whether the magnitudes $|\WHT_f(u)|$ satisfy the integer-valued property required by our characterization theorems for generalized bent, generalized plateaued, or landscape functions.

Standard cipher S-boxes, when viewed as generalized Boolean functions $f: \mathbb{F}_2^n \to \mathbb{Z}_{2^k}$, exhibit non-integer Walsh magnitude values and therefore do not satisfy the characterizations established in this paper.

\begin{example}[PRESENT S-box analysis]
\label{ex:present_analysis}
The PRESENT S-box~\cite{present}, viewed as $f: \mathbb{F}_2^4 \to \mathbb{Z}_{16}$,
has a Walsh spectrum containing $15$ distinct magnitude values. A sample of these
magnitudes includes
$|\WHT_f(0)| = 0.000$,  $|\WHT_f(1)| \approx 1.104$,  $|\WHT_f(2)| \approx 1.298$, 
$|\WHT_f(3)| \approx 2.054$,  $|\WHT_f(4)| \approx 2.299$,  $|\WHT_f(5)| = 2\sqrt{2}$.
This S-box was designed to optimize Boolean cryptographic properties (nonlinearity, differential uniformity, algebraic immunity), not the specific component relationships
required by our characterization theorems. By Theorem~\textup{\ref{reg-landscape}}, any landscape function over
$\Znk$ must have Walsh magnitudes in
$\{2^{m/2}v : m\in\N_0,\, v\in 2\N_0+1\}$.
The observed non-integer magnitudes confirm that $f$ is not landscape, and in particular is neither gbent nor gplateaued.
Moreover, by the contrapositive of Theorem~\textup{\ref{thm:landscape_necessity}},
there must exist some  
$F: (\Znl)^{r-1} \to \Z_{2^m}$ for which 
$f_F(x) = c_{r-1}(x) + 2^{\ell-m}\,F(c_0(x),\ldots,c_{r-2}(x))$ fails to be landscape over $\Znl$ with the same magnitudes as $f$, confirming that the component relationships required by our
characterization theorems are absent in Boolean-optimized designs.
\end{example}

Interestingly, both PRESENT and GIFT S-boxes exhibit the Walsh magnitude
$2\sqrt{2}$, a quadratic algebraic number. This value is the only
``simple'' algebraic number appearing in their spectra; all other magnitudes are
higher-degree algebraic numbers from the cyclotomic field $\mathbb{Q}(\zeta_{16})$.
The appearance of $2\sqrt{2} = 2^{3/2}$ may be related to half-bent structure, as
it lies between the fully bent magnitude $2^{n/2} = 4$ and the plateaued structure.

The failure of classical S-boxes to satisfy the generalized landscape property has
a clear algebraic explanation. Standard cipher S-boxes are designed as Boolean
vectorial functions $(f_0, \ldots, f_{k-1}): \mathbb{F}_2^n \to \mathbb{F}_2^k$,
where each component $f_i$ is optimized independently for cryptographic properties
such as high nonlinearity (resistance to linear cryptanalysis), low differential
uniformity (resistance to differential cryptanalysis), and algebraic complexity
(resistance to algebraic attacks). These design criteria are orthogonal to the
spectral structure required by our theorems.

Specifically, for a function $f: \mathbb{F}_2^n \to \mathbb{Z}_{2^k}$ to be
generalized bent, Corollary~\ref{cor:iff} requires $1$-sparsity of the partition
coefficients $C_\alpha(u)$ for each $u \in \V$: exactly one coefficient can be
nonzero. For $f$ to be $s$-gplateaued, Theorem~\ref{thm:plateaued_common_arg}
requires the nonzero coefficients to satisfy the common-argument hypothesis and the
derived functions $g_\beta$ to be $s$-gplateaued over $\Znl$. For $f$ to be
landscape, Theorem~\ref{thm:landscape_necessity} requires  at least every function in the
affine space $c_{r-1} + \langle c_0, \ldots, c_{r-2}\rangle$ to be landscape over
$\Znl$ with identical Walsh magnitudes. None of these conditions are imposed by
classical S-box design methodologies.

Furthermore, when viewed as a function to $\mathbb{Z}_{2^k}$, the Walsh transform
involves $2^k$-th roots of unity, so the magnitudes $|\WHT_f(u)|$ are in general
algebraic numbers of high degree over $\mathbb{Q}$. For these magnitudes to belong
to the set $\{2^{m/2}v : m \in \N_0, v \in 2\N_0+1\}$ as required for landscape
functions (Theorem~\ref{reg-landscape}), the phases $\zeta_{2^k}^{f(x)}$ must align
in very specific geometric configurations---precisely the configurations guaranteed
by our construction theorems via the $2^\ell$-adic decomposition, but absent in
Boolean-optimized S-boxes. This confirms that the $2^\ell$-adic perspective provides
genuinely new construction tools, complementary to classical Boolean design
methodology.

\section{Conclusion and future research directions}
\label{sec:conclusion}

We have introduced the $2^\ell$-adic decomposition as a powerful technique for
analyzing and constructing generalized Boolean functions with prescribed cryptographic
properties, combining additive combinatorics with Galois theory of cyclotomic
extensions.

Our foundational contribution is Theorem~\ref{thm:overconstrained}, proving
unconditionally that two-level magnitude spectra force extreme sparsity under the
common-argument hypothesis, with Theorem~\ref{thm:overconstrained_multilevel}
providing a conditional multi-level extension. 
Our main results on the $2^\ell$-adic decomposition are  Theorem~\ref{thm:landscape_necessity} and Theorem~\ref{thm:landscape_iff}.  Theorem~\ref{thm:landscape_necessity} establishes
unconditionally that if $f: \V \to \Znk$ is landscape, then every function in the
affine space $c_{r-1} + \langle c_0, \ldots, c_{r-2} \rangle$ is landscape over $\Znl$ with identical Walsh magnitudes, and, furthermore,  in general,  any function  $f_F(x) = c_{r-1}(x) + 2^{\ell-m}\,F(c_0(x),\ldots,c_{r-2}(x))$, where $F: (\Znl)^{r-1} \to \Z_{2^m}$, for any $1\leq m\leq\ell$, is landscape over $\Znl$ with identical Walsh magnitudes. A complete characterization via a small subset of maps
$f_F$
is established in Theorem~\ref{thm:landscape_iff}, and a spectral sparsity interpretation in Corollary \ref{cor:sparsity}. Sufficiency
under natural structural assumptions is established 
when the nonzero partition coefficients satisfy the common-argument hypothesis (Theorem~\ref{thm:plateaued_common_arg}
for gplateaued functions). For gbent functions, Corollary~\ref{cor:main} and and Proposition~\ref{thm:iff_basis} synthesize these into a complete (conditional) characterization, Corollary~\ref{cor:iff} provides the spectral
interpretation as $1$-sparsity of the coefficients $C_\alpha(u)$.
Example~\ref{ex:counterexample_3to1} demonstrates that these structural assumptions are genuinely necessary.

We showed that the decomposition does not preserve the Maiorana--McFarland class (Theorem~\ref{thm:MM})
but that it preserves duality
(Proposition~\ref{prop:dual}), quadraticity (Corollary~\ref{cor:second_order}), and enables precise differential analysis (Theorem~\ref{thm:differential_affine}).
Our analysis of classical cipher S-boxes (Section~\ref{sec:classical-sboxes})
confirms that standard Boolean-optimized designs exhibit non-integer Walsh magnitudes, validating that the $2^\ell$-adic perspective provides genuinely new construction tools complementary to classical methodology.

Several research directions emerge. From a constructive perspective, developing
systematic algorithms to build gbent and gplateaued functions exploiting the
$2^\ell$-adic structure would enhance practical utility, particularly for
implementations where arithmetic modulo $2^\ell$ is more natural than bit operations.
The connections to algebraic number theory merit further exploration: Can gbent
properties be characterized via ideal structures in cyclotomic rings $\Z[\zeta_{2^k}]$?
Do different tower extensions $\Q(\zeta_{2^\ell}) \subset \Q(\zeta_{2^k})$ correspond
to different cryptographic properties? Extensions to $p$-adic decompositions for odd
primes $p$, to vectorial generalized bent functions, and to a deeper understanding
of when the common-argument hypothesis can be verified efficiently represent
important open directions. Finally, while Theorem~\ref{thm:landscape_iff} provides a complete
characterization of landscape functions without structural assumptions via a small subset of the family of maps $F$, a sufficiency theorem based on the \emph{affine subfamily} $\{f_\beta\}_{\beta\in(\Znl)^{r-1}}$ alone --- without additional hypotheses on the lower components --- remains open.
Example~\ref{ex:counterexample_3to1} shows such a result is impossible in full generality, but identifying the weakest sufficient condition that bridges the affine subfamily and the full family is an intriguing open problem.

The $2^\ell$-adic decomposition reveals that generalized bent, plateaued, and
landscape functions possess rich hierarchical structure previously hidden by binary
representation. By exploiting this structure through sparsity arguments and
cyclotomic field theory, we have provided both theoretical insights and practical
verification tools. The interplay between algebraic structure, spectral properties,
and cryptographic utility emerges as a central theme: necessity results show that
spectral constraints force algebraic structure, while sufficiency results demonstrate
that algebraic structure guarantees spectral properties.

\vspace{.5cm}

\noindent
\section*{Acknowledgements} The authors would like to thank the organizers of the Bent Camp 2025 at Sabanc{\i} University for making it possible for us to meet and start working on the problems we discussed in this paper.

\end{document}